\documentclass[12pt,twoside,reqno]{amsart}
\usepackage[colorlinks=true,citecolor=blue]{hyperref}
\usepackage{mathptmx, amsmath, amssymb, amsfonts, amsthm, mathptmx, enumerate, color,mathrsfs}
\usepackage{graphicx}

\usepackage{multirow}
\usepackage{epstopdf}
\usepackage{multicol}
\usepackage{algorithm}
\usepackage{algorithmic}
\usepackage{epstopdf}

\usepackage{latexsym,bm,amsmath,amssymb,mathrsfs}
\usepackage{rotating}
\usepackage{multirow,booktabs,cases}
\usepackage[style=1]{mdframed}
\usepackage{algorithm,algorithmic}
\usepackage[colorlinks=true]{hyperref}
\hypersetup{urlcolor=blue,citecolor=blue,linkcolor=blue}
\usepackage{makecell}

\usepackage[justification=centering]{caption}

\def\[{\begin{equation}}
\def\]{\end{equation}}

\newcommand{\RR}{\mathbb{R}}
\newcommand{\levelset}{\mathrm{\bf Lev}} 
\newcommand{\dom}[1]{\mathrm{\bf dom}\,{(#1)}} 
\newcommand{\crit}[1]{\mathrm{\bf crit}\,{(#1)}} 
\newcommand{\prox}{\mathrm{\bf Prox}} 
\newcommand{\dist}{\mathrm{\bf dist}} 
\newcommand{\intset}[1]{\mathrm{\bf int}\,{(#1)}} 

\newcommand{\wxx}{\widehat{x}}
\newcommand{\xii}{\xi}
\newtheorem{assumption}{Assumption}

\def\B{\mathscr{B}}

\def\X{{\mathcal{X}}}

\newcommand{\wo}{\widehat{\omega}}
\newcommand{\wO}{\widehat{\Omega}}
\newcommand{\nn}{\nonumber}

\def\nn{\nonumber}

\def\B{{\mathscr B}}

\newtheorem{theorem}{Theorem}[section]
\newtheorem{lemma}[theorem]{Lemma}

\theoremstyle{definition}
\newtheorem{definition}[theorem]{Definition}

\newtheorem{remark}[theorem]{Remark}
\numberwithin{equation}{section}

\graphicspath{{./FIG/}}

\begin{document}
\setcounter{page}{1}

\vspace*{2.0cm}
\title[An Inertial Bregman Proximal DC Algorithm for Generalized DC Programming]
{An Inertial Bregman Proximal DC Algorithm for Generalized DC Programming with Application to Data Completion}
\author[Chenjian Pan, Yingxin Zhou, Hongjin He, Chen Ling]{ Chenjian Pan$^{1}$, Yingxin Zhou$^{1}$, Hongjin He$^{1*}$, Chen Ling$^{2}$}
\maketitle
\vspace*{-0.6cm}

\begin{center}
{\footnotesize

$^1$School of Mathematics and Statistics, Ningbo University, Ningbo, 315211, China.\\
$^2$Department of Mathematics, Hangzhou Dianzi University, Hangzhou, 310018, China.

}\end{center}

\vskip 4mm {\footnotesize \noindent {\bf Abstract.}
In this paper, we consider a class of generalized difference-of-convex functions (DC) programming, whose objective is the difference of two convex (not necessarily smooth) functions plus a decomposable (possibly nonconvex) function with Lipschitz gradient. By employing the Fenchel-Young inequality and Moreau decomposition theorem, we introduce an inertial Bregman proximal DC algorithm to solve the problem under consideration. Our algorithmic framework is able to fully exploit the decomposable structure of the generalized DC programming such that each subproblem of the algorithm is enough easy in many cases. Theoretically, we show that the sequence generated by the proposed algorithm globally converges to a critical point under the Kurdyka-{\L}ojasiewicz condition. A series of numerical results demonstrate that our algorithm runs efficiently on matrix and tensor completion problems.

 \noindent {\bf Keywords.}
Difference-of-convex programming; Fenchel-Young inequality; Moreau decomposition; Bregman distance; tensor completion.

 \noindent {\bf 2020 Mathematics Subject Classification.}
65K10, 90C26. }

\renewcommand{\thefootnote}{}
\footnotetext{ $^*$Corresponding author.
\par
E-mail address: hehongjin@nbu.edu.cn (Hongjin He).
\par

}

\section{Introduction}\label{Introduction}

In this paper, we are interested in a class of generalized difference-of-convex functions (DC) programming, which refers to
\begin{align}\label{priminal problem}
\min_{x\in\RR^n}~\Phi(x) := f(x)-g(x)+h(x),
\end{align}
where $f(x)$ and $g(x)$ are proper closed convex functions, and $h(x)$ is a differentiable (possibly nonconvex) function, which is further assumed to be liberally decomposable in the sense of $h(x)=h^+(x)-h^-(x)$ with Lipschitz continuity modulus $L_h^+$ and $L_h^-$ for $h^+(x)$ and $h^-(x)$, respectively. The generalized DC programming \eqref{priminal problem} has widespread applications in signal/image processing and statistical/machine learning \cite{AEB06,LTPD15,LTPD18,SHWJ23,YLHX15}, to name just a few.

When the last decomposable part $h(x)$ in \eqref{priminal problem} vanishes, a benchmark solver to circumvent the nonconvexity of the resulting problem is the so-named DC algorithm (DCA) originally proposed by Pham Dinh \cite{PDS86} (see also \cite{PDLT97}), where $-g(x)$ is replaced by its linear approximation at $x^k$ so that the DCA enjoys a convex subproblem. Naturally, a straightforward application of the DCA to \eqref{priminal problem} produces the following iterative scheme:
\begin{align}\label{DCA-one}
x^{k+1}\in\arg\min_{x \in \RR^n}~\left\{f(x)+h(x)-\langle \xi^{k+1} ,x-x^k\rangle\right\},
\end{align}
where  $\xi^{k+1}$ is a subgradient of $g(x)$ at the iterate $x^k$. However, the iterative scheme \eqref{DCA-one} is not necessarily implementable in practice due to the simultaneous appearance of $f(x)$ and $h(x)$, especially the possible nonconvexity of $h(x)$. Hence, one more reasonable way is to exploit the differentiability of $h(x)$. In this way, a customized application of the DCA to \eqref{priminal problem} yields the following iterative scheme:
\begin{align}\label{subproblem-DCA}
x^{k+1}\in\arg\min_{x \in \RR^n}~\left\{f(x)+\langle \nabla h(x^k)-\xi^{k+1} ,x-x^k\rangle\right\},
\end{align}
where $\nabla h(x^k)$ is the gradient of $h(x)$ at the iterate $x^k$. Clearly, \eqref{DCA-one} and \eqref{subproblem-DCA} present two different applications of the original DCA. However, \eqref{subproblem-DCA} is more implementable than \eqref{DCA-one} in many cases. To some extent, we can regard the seminal DCA as an algorithmic paradigm. Therefore, how to exploit the DC decomposition in algorithmic design and implementation is of importance for efficiently solving DC programming and general nonconvex optimization problems \cite{PDLT97}. Consequently, as studied in \cite{GTT17}, we can also reformulate \eqref{priminal problem} as a canonical DC programming: 
\begin{align}\label{nDCR}
\min_{x\in\mathbb{R}^n}\; \Phi(x) := \underbrace{\left( f(x)+\frac{L_h}{2}\|x\|^2 \right)}_{\widehat{f}(x)}-\underbrace{\left( g(x)-h(x)+\frac{L_h}{2}\|x\|^2 \right)}_{\widehat{g}(x)},
\end{align}
where $L_h$ is the Lipschitz continuity constant of $h(x)$ so that both $\widehat{f}(x)$ and $\widehat{g}(x)$ are convex functions.
Then, directly applying the DCA to \eqref{nDCR} leads to 
\begin{align}\label{subproblem-PDCA}
x^{k+1}=\arg\min_{x \in \RR^n}~\left\{f(x)+\frac{L_h}{2}\left\|x-x^k+\frac{1}{L_h}\left(\nabla h(x^k)-\xi^{k+1}\right)\right\|^2\right\},
\end{align}
which shares the same idea of the proximal DCA introduced in \cite{SSC03}. Moreover, such a scheme immediately reduces to the classical proximal gradient method \cite{BT09} when $g(x)$ vanishes. Comparatively, \eqref{subproblem-PDCA} is more practical than \eqref{DCA-one} and \eqref{subproblem-DCA} for dealing with nonsmooth optimization problems. Although \eqref{subproblem-PDCA} enjoys an easier subproblem, it is essentially a gradient-like algorithm which runs slowly in many cases (e.g., see \cite{OC15} for the discussion on convex optimization). Therefore, many researchers are motivated to develop faster algorithms for DC programming, e.g., see  \cite{CHZ22,LPT19,LZ19,LZS19,PLT18,PLT23,TFT22,WCP18} and references therein. 

In this paper, we follow the extrapolation idea and the Bregman proximal technique to propose an inertial Bregman Proximal DC Algorithm (iBPDCA) for the generalized DC programming \eqref{priminal problem}. It is noteworthy that most DCA-like algorithms must select a subgradient $\xi^{k+1}$, which is an element of the subdifferential $\partial g(x^k)$ when dealing with nonsmooth DC programming. In this situation, it is not an easy task to determine which one is the most ideal subgradient for approximating $-g(x)$. Therefore, to circumvent this difficulty, we accordingly employ the well known Fenchel-Young inequality to introduce a surrogate for $-g(x)$. Then, we add a proximal term to enhance the subproblem for updating $\xi^{k+1}$. One direct benefit is that we will no longer worry about the selection of the subgradient since it will be updated uniquely as long as $g(x)$ is convex. Another remarkable benefit is that we can directly exploit the possibly explicit proximal operator of $g(x)$ by the extended Moreau decomposition theorem \cite{Beck17}, which is able to make our algorithm more implementable in practice (see Section \ref{Sec:Num}).  Moreover, we shall highlight that our algorithm also possesses an extrapolation step for the algorithmic acceleration, which will be verified in Section \ref{Sec:Num}. Thirdly, we should mention that $h(x)$ is assumed to be liberally decomposable of the form $h(x)=h^+(x)-h^-(x)$. Such a decomposable form allows us to freely reformulate the DC optimization model so that which, together with the Bregman proximal term, can gainfully help us design customized algorithms for solving the problem under consideration. Theoretically, we show that the sequence generated by the proposed iBPDCA converges to a critical point of \eqref{priminal problem}. Finally, we apply the iBPDCA to low-rank matrix and tensor data completion. A series of computational results demonstrate that our algorithm works well in solving DC optimization problems.

The structure of this paper is divided into five parts. In Section \ref{Sec:Pre}, we summarize some notations and recall some basic definitions and properties. In Section \ref{Sec:Alg}, we introduce our iBPDCA algorithm for \eqref{priminal problem} and study its convergence.  In Section \ref{Sec:Num}, we conduct the numerical performance of our iBPDCA on data completion. Finally, some concluding remarks are summarized in Section \ref{Sec:Con}.

\section{Preliminaries} \label{Sec:Pre}
In this section, we briefly review several basic notions, definitions, and properties 
 of Kurdyka-{\L}ojasiewicz inequality and other related mathematical tools that will be used throughout this paper.

We denote $\RR^n$ as the $n$-dimensional Euclidean space with inner product $\langle \cdot,\cdot\rangle$, and we use $\|\cdot\|_1,\; \|\cdot\|,\; \|\cdot\|_*$ and $\|\cdot\|_F$ to denote the $\ell_1$ norm, $\ell_2$ norm, nuclear norm and  Frobenius norm for vectors or matrices, respectively. For any subset $\mathbb{S}\subset \RR^n$ and any point $x\in\RR^n$, $\dist(x,\mathbb{S})=\inf\,\left\{ \| y-x\|\; |\; y\in\mathbb{S}\right\}$ is used to define the distance from $x$ to $\mathbb{S}$. As for a matrix $X\in\mathbb{R}^{m\times n}$, the transpose of $X$ is denoted by $X^\top$.

For an extended-real-value function $f:\RR^n\to[-\infty,+\infty]$, we denote the domain and level set of $f$ by $\dom{f} := \left\{x\in \RR^n\;|\;f(x) <\infty\right\}\; \text{and}\; \levelset_f(\alpha_k)=\{x\in \RR^n\;|\; f(x) \le \alpha_k\},$ respectively. It is documented in \cite{RW98} that a proper closed function $f$ is said to be level-bounded if $\levelset_f(\alpha_k)$ is bounded. Moreover, a proper function is closed if it is lower semicontinuous.
\begin{definition}\label{def:subdiff}
	Let $f: \RR^n \to (-\infty,\infty]$ be a proper and lower semicontinuous function.
	\begin{enumerate}
		\item[\rm (i)] For each $x \in \dom{f}$, the Fr\'{e}chet subdifferential $\widehat{\partial}f(x)$ of $f$ at $x$ is defined by
		\begin{equation*}
		\widehat{\partial}f(x)  :=  \left\{ \xi\in\RR^n \;\Big{|}\; \liminf_{\substack{y\ne x\\ y\to x}} \frac{f(y)-f(x) -\langle \xi,y-x \rangle}{\|y-x\|} \ge 0\right\}.
		\end{equation*}
		In particular, for $x \notin \dom{f}$, we set $\widehat{\partial}f(x)  = \emptyset$.
		\item[\rm (ii)] The limiting subdifferential $\partial f(x)$ of $f$ at $x\in \dom{f}$ is defined by
		\begin{equation*}
		\partial f(x)  := \left\{\xi \in \RR^n\;\Big{|}\; \begin{array}{l}
		\exists \left(x^k,f(x^k)\right) \to \left(x,f(x)\right), \xi^k \in \widehat{\partial}f(x^k) \\ \text{such that  }\; \xi^k\to \xi\; \text{  as } \;k\to \infty 
		\end{array} \right\}.
		\end{equation*}
	\end{enumerate}
\end{definition}
If $f$ is continuously differentiable, then $\partial f=\{\nabla f\}$ (e.g., see \cite[Exercise 8.8(b)]{RW98}), where $\nabla f$ is denoted as the gradient of $f$. Moreover, when $f$ is convex, the limiting subdifferential reduces to the classical subdifferential in convex analysis (see \cite[Proposition 8.12]{RW98}).
Apart from that, Definition~\ref{def:subdiff} implies $\widehat{\partial}f  \subset \partial f $ for each $x \in \RR^n$. It is worth noting that the set $\widehat{\partial}f $ is both convex and closed while $ \partial f$ is closed (e.g., see \cite[Theorem 8.6]{RW98}). 

A necessary but not sufficient condition for $x^{\star} \in \RR^n$ to be a local minimizer of $f$ is that $0 \in \partial f(x^{\star})$, where $x^{\star}$ is called a stationary point of $f$. Throughout this paper, the set of stationary points of $f$ is denoted by $\Xi({f})$. Additionally, as defined in the literature (e.g., \cite{LTPD18,WCP18}), we say that $x^\star \in \RR^n$ is a critical point of $\Phi:=f-g+h$ given in \eqref{priminal problem} if 
\begin{equation}\label{def:cp}
0\in \partial f(x^\star)-\partial g(x^\star)+\nabla h(x^\star),
\end{equation}
or equivalently, $\left\{\partial f(x^\star)+\nabla h(x^\star)\right\}\cap \partial g(x^\star)\neq \emptyset$. Then, the set of critical points of $\Phi$ is denoted by $\crit{\Phi}$. Here, we should emphasize that a stationary point of $\Phi$, i.e., $0\in \partial(f-g)(x^\star)+\nabla h(x^\star)$, is sharper than a critical point of $\Phi$ satisfying \eqref{def:cp}. If $g$ is assumed to be continuously differentiable over $\dom{f+h}$, i.e., $\partial g(x^\star)=\{\nabla g(x^\star)\}$, then a critical point of $\Phi$ is also a stationary point of $\Phi$ (see \cite{deO20}). In our paper, we do not assume $g$ being a smooth function and adopt the critical point \eqref{def:cp} for simplicity.

Before stating the Kurdyka-{\L}ojasiewicz (K{\L}) property, we first use $\varUpsilon_\zeta$ to denote the set of all concave continuous functions $\phi : [0,\zeta) \to \RR_+$ that are continuously differentiable on $(0, \zeta)$ with positive derivatives and satisfy $\phi(0) =0$.
Below, we recall the widely used K{\L} property (e.g., see \cite[Definition 3.1]{ABRS10}) that plays an instrumental role in convergence analysis of many nonconvex optimization methods.
\begin{definition}[K{\L} property and K{\L} function]\label{def:KL}
	Let $f: \RR^n \to (-\infty,\infty]$ be a proper and lower semicontinuous function.
	\begin{enumerate}
		\item [\rm (i)] We say that the function $f$ has the K{\L} property at 
		$\bar{x} \in \dom{\partial f}$
		if there exist $\zeta \in (0,\infty]$, a neighborhood $\mathcal{N}$ of $\bar{x}$ and a continuous concave function $\phi\in \varUpsilon_\zeta$ such that for all $x \in \mathcal{N}$ with $f(\bar{x}) < f(x) < f(\bar{x}) + \zeta $, the following K{\L} inequality holds, i.e.,
		\begin{equation*}
		\phi^\prime \left(f(x)-f(\bar{x})\right) \dist\left(0,\partial f(x)\right) \ge 1.
		\end{equation*}
		\item [\rm (ii)] If $f$ satisfies the above K{\L} property at each point in $\dom{\partial f}$, then $f$ is called a K{\L} function.
	\end{enumerate}
\end{definition}
Hereafter, we recall the uniformized K{\L} property introduced in \cite[Lemma 6]{BST14}.
\begin{lemma}[uniformized K{\L} property]\label{lem:UKL}
	Let $\Gamma$ be a compact set and let $f: \RR^n \to (-\infty,\infty]$ be a proper and lower semicontinuous function. Assume that $f$ is constant on $\Gamma$ and satisfies the K{\L} property at each point of $\Gamma$. Then, there exist $\varsigma>0$, $\zeta>0$ and $\phi\in \varUpsilon_\zeta$ such that for each $\bar{x} \in \Gamma$ the following K{\L} inequality
	\begin{equation*}
	\phi^\prime\left(f(x)-f(\bar{x}) \right) \dist \left(0,\partial f(x)\right) \ge 1
	\end{equation*}
	holds for any
	$	x \in \left\{ x\in \RR^n\;|\; \dist (x,\Gamma) <\varsigma \right\} \cap \left\{ x \in \RR^n\;|\; f(\bar{x}) < f(x) < f(\bar{x}) + \zeta\right\}.$
\end{lemma}

Note that we assume $h(x)$ being differentiable. Here, we recall the famous descent lemma for the so-called $L$-smooth optimization problems (see  more details in \cite[Lemma 5.7]{Beck17}).
\begin{lemma}\label{lem:delem}
	Let $f:\RR^n\to(-\infty,\infty]$ be an $L$-smooth function  ($L\geq 0$) 
	over a given convex set $\mathbb{S}$, i.e., $f$ is differentiable over $\mathbb{S}$ and satisfies
	$$\| \nabla f(x) -\nabla f(y)\| \leq L \|x-y\|, \quad \forall x,y \in \mathbb{S}.$$
	Then, for any $x,y\in\mathbb{S}$, we have 
	$$f(y)\leq f(x) + \langle \nabla f(x),y-x\rangle + \frac{L}{2}\|x-y\|^2.$$
\end{lemma}

Now, we present the definition of conjugate function, which plays a significant role in finding a surrogate of $-g(x)$ for our algorithmic design.
\begin{definition}\label{def:conf}
	Let $f:\RR^n \to [-\infty,\infty]$ be an extended real-valued function. The conjugate function of $f$ is defined by
	\begin{equation*}
	f^*(y) = \sup_{x\in \RR^n} \left\{ \langle y, x \rangle - f(x) \right\}, \; y \in \RR^n.
	\end{equation*}
\end{definition}
From \cite[Chapter 4]{Beck17}, $f^*$ is proper, lower semicontinuous and convex provided $f$ is proper, lower semicontinuous and convex. Furthermore, Definition \ref{def:conf} implies
\begin{equation}\label{eq:FY}
f(x)  + f^*(y) \ge \langle x,y \rangle
\end{equation}
holds. Especially, the equality holds if and only if $y \in \partial f(x) $. Besides, if $f$ is proper, closed and convex, then for any $x$ and $y$, one has $y \in \partial f(x) $ if and only if $x \in \partial f^*(y)$ (e.g., see \cite[Theorem 4.20]{Beck17}). In the literature, \eqref{eq:FY} is often referred to as the Fenchel-Young inequality, which will be vital for our algorithmic design.

\begin{definition}[proximal mapping]\label{def:prox}
	Given a function $\vartheta :\RR^n\to (-\infty,\infty]$ and a scalar $t>0$, the proximal mapping of $t\vartheta $ is the operator given by
	\begin{equation}\label{eq:prox}
	\prox_{t\vartheta }({\bm a})=\arg\min_{x\in\RR^n}\left\{t\vartheta (x)+\frac{1}{2}\|x-{\bm a}\|^2\right\} \quad \text{for any } {\bm a}\in\RR^n.
	\end{equation}
\end{definition}

\begin{theorem}[extended Moreau decomposition theorem]\label{the:moreau}
	Let $f: \RR^n\to (-\infty,\infty]$ be proper closed and convex, and let the scalar $t>0$. Then for any $x\in \RR^n$, we have
	\begin{align*}
		\prox_{tf}(x)+t\prox_{t^{-1}f^*}\left(\frac{x}{t}\right) = x.
	\end{align*}
\end{theorem}
This theorem connects the proximal operator of proper closed convex functions and their
conjugates.

We conclude this section by introducing the definition of Bregman distance and its associated properties (see \cite{Bregman-1967} and \cite[Chapter 9]{Beck17}). These concepts will serve as valuable tools for our algorithmic design and convergence analysis.
\begin{definition}[Bregman distance]\label{def:Bregman}
	Let $\psi: \RR^n \to (-\infty, + \infty]$ be a strongly convex and continuously differentiable function over $\dom{\partial \psi}$. The Bregman distance associated with the kernel $\psi$ is the function $\B_\psi:\dom{\psi}\times \dom{\partial \psi}\to \RR$ given by
	\begin{equation*}
	\B_\psi(x,y) = \psi(x) - \psi(y) - \langle \nabla \psi(y), x-y\rangle,\quad \forall x \in \dom{\psi}, y \in \intset{\dom{\psi}}.
	\end{equation*}
\end{definition}

\begin{lemma}\label{lem:Bregman}
	Suppose that $\mathbb{S}\subseteq \RR^n$ is nonempty, closed and convex, and the function $\psi$ is proper closed convex and differentiable over $\dom{\partial \psi}$. If $\mathbb{S}\subseteq \dom{\psi}$ and $\psi+\mathcal{I}_{\mathbb{S}}$ with $\mathcal{I}_{\mathbb{S}}$ being an indicator function associated with $\mathbb{S}$ is $\varrho$-strongly convex ($\varrho>0$), then the Bregman distance $\B_\psi$ associated with $\psi$ has the following properties:
	\begin{itemize}
		\item[\rm (i)] $\B_{\psi}(x,y)\geq \frac{\varrho}{2}\|x-y\|^2$ for all $x\in\mathbb{S}$ and $y\in\mathbb{S}\cap \dom{\partial \psi}$;
		\item[\rm (ii)] Let $x\in\mathbb{S}$ and $y\in\mathbb{S}\cap \dom{\partial \psi}$. Then $\B_{\psi}(x,y)\geq 0$, and in particular, the equality holds if and only if $x=y$;
		\item[\rm (iii)] If $\nabla \psi$ is Lipschitz continuous with modulus $L_\psi$, it holds that $\B_{\psi}(x,y)\leq \frac{L_\psi}{2}\|x-y\|^2$ for all $x\in\mathbb{S}$ and $y\in\mathbb{S}\cap \dom{\partial \psi}$.
	\end{itemize}
\end{lemma}

\section{Algorithm and convergence analysis}\label{Sec:Alg}
In this section, we aim to introduce a new algorithm for \eqref{priminal problem} and  establish the convergence result for the proposed algorithm.
\subsection{Algorithm: iBPDCA}
The inertial technique has been widely used for algorithmic acceleration. Therefore, we first follow the inertial idea to generate an intermediate point via 
\begin{equation}\label{extrastep}
\wxx^k = x^k+\alpha_k(x^k-x^{k-1}),
\end{equation}
where $\alpha_k\in(0,1]$ is an extrapolation parameter. Then, we employ the Fenchel-Young inequality \eqref{eq:FY} to iteratively construct a majorized surrogate of $-g(x)$ at $\wxx^k$:
\begin{align*}
g^*(\xi)-\langle \xi,\wxx^k\rangle \geq -g(\wxx^k), \quad \forall \xi\in\mathbb{R}^n.
\end{align*}
In this situation, we can construct a $\xi$-subproblem to update $\xi^{k+1}$ uniquely by attaching a proximal term $\frac{\beta}{2}\|\xi-\xi^k\|^2$, i.e.,
\begin{align}\label{eq:subgx}
\xi^{k+1} = \arg\min_{\xi \in \RR^n} \left\{g^*(\xi) - \langle \wxx^k, \xi \rangle +\frac{\beta}{2}\|\xi-\xi^k\|^2\right\}, 
\end{align}
where $\beta>0$ is a regularization parameter. Finally, we invoke the decomposable form of $h(x)=h^+(x)-h^-(x)$, and reformulate \eqref{priminal problem} as
\begin{align*}
\min_{x\in \mathbb{R}^n}\Phi(x) = \left(f(x)+h^+(x)\right)-\left(g(x)+h^-(x)\right).
\end{align*}
Consequently, we follow the DCA spirit and attach a Bregman proximal term to the $x$-subproblem. Specifically, we update $x^{k+1}$ via 
\begin{align}\label{algeq:update_x}
x^{k+1}=& \arg\min_{x \in \RR^n} \left\{f(x)+h^+(x)- \langle x-\wxx^k, u^k\rangle+ \tau\B_{\psi}(x,\wxx^k)\right\},
\end{align}
where $u^k=\xi^{k+1} + \nabla h^-(\wxx^k)$ and $\tau>0$ is also a proximal regularization parameter. Concretely, we summarize the details for solving \eqref{priminal problem} in Algorithm \ref{alg:proposed}.

\begin{algorithm}
	\caption{Inertial Bregman Proximal DC Algorithm (iBPDCA) for solving \eqref{priminal problem}.}\label{alg:proposed}
	\begin{algorithmic}
		\STATE{Choose a starting point $x^0,~{\rm set}~x^{-1}=x^0, \{\alpha_k\}\subseteq [\alpha_{\min},1], 0<\alpha_{\min}<1, \beta>1/2~{\rm and~}\tau>0.$}
		\FOR{$k=0,1,2,\cdots,$}
		\STATE{Calculate $\wxx^k$ via \eqref{extrastep}.}
		\STATE{Compute $\xi^{k+1}$via \eqref{eq:subgx}.}
		\STATE{Update $x^{k+1}$ via \eqref{algeq:update_x}.
		}
		\ENDFOR
	\end{algorithmic}
\end{algorithm}

\begin{remark}
	In the literature, most existing DC-type algorithms usually randomly select a subgradient of $g(x^k)$ when $g$ is not differentiable. Comparatively, we can obtain a unique $\xi^{k+1}$ via \eqref{eq:subgx} as long as $g(x)$ is a convex function. Actually, updating $\xi^{k+1}$ amounts to evaluating $\prox_{\frac{1}{\beta}g^*}\left(\xi^k+\frac{1}{\beta}\widehat{x}^k\right)$, which, with the help of the extended Moreau Decomposition Theorem \ref{the:moreau}, can be easily obtained via
	\begin{align*}
	\prox_{\frac{1}{\beta}g^*}\left(\xi^k+\frac{1}{\beta}\widehat{x}^k\right) =  \xi^k+\frac{1}{\beta}\widehat{x}^k-\frac{1}{\beta}\prox_{{\beta}g}\left(\beta\xi^k+\widehat{x}^k\right).
	\end{align*}
Therefore, we do not care about the explicit form of $g^*(x)$, and our algorithm enjoys a simpler iterative scheme.
\end{remark}

\subsection{Convergence analysis}
In this subsection, we are concerned with the convergence properties of Algorithm \ref{alg:proposed}.
To begin with, we state some standard assumptions on \eqref{priminal problem}, which will be used in convergence analysis.
\begin{assumption}\label{ass:1}
	$f(x)$ is a  proper lower semicontinuous function and $g(x)$ is a continuous and convex function.	
\end{assumption}
\begin{assumption}\label{ass:2}
	$\inf_{x}\Phi(x) :=\inf_{x}\{ f(x)-g(x)+h(x)\} > -\infty.$
\end{assumption}
\begin{assumption}\label{ass:3}
	$h(x)$ is decomposable of the form $h(x):=h^+(x)-h^-(x)$ with Lipschitz continuity modulus $L_h^+$ and $L_h^-$ for $h^+(x)$ and $h^-(x)$, respectively. 
\end{assumption}
In the coming analysis, for notational simplicity, we denote
\begin{align*}
\Theta(x,\xi)= f(x)+g^*(\xi)-\langle \xi,x\rangle+h^+(x)-h^-(x)
\end{align*}
as a surrogate objective function to approximate $\Phi(x)$. From the Fenchel-Young inequality, it is easy to get $\Theta(x,\xi)\geq \Phi(x).$

\begin{lemma}\label{phi-decrease-lemma-conclusion1}
	Suppose that Assumptions~\ref{ass:1},~\ref{ass:2} and \ref{ass:3} hold. Then, there exist two scalars $H$ and $P$ such that 
	\begin{align}\label{phi-decrease-lemma-conclusion}
	\Theta(x^{k+1},\xi^{k+1})\le & \Theta(x^k,\xi^{k})+ H\|\wxx^k-x^k\|^2+ P\|\wxx^{k+1}-x^{k+1}\|^2
	+\left(\frac{1}{2}-\beta\right)\|\xi^{k+1}-\xi^k\|^2,
	\end{align}
	where
	\begin{align}\label{HPeq}
	H=\frac{\tau L_{\psi}+1+2(L_h^-)^2-2\tau\rho}{2}\quad \text{and} \quad
	P=\frac{2(L_h^-)^2+L_h^- -2\tau\rho+1}{{2\alpha_{\min}^2}}.
	\end{align} 
\end{lemma}

\begin{proof}
	From the first-order optimality condition of \eqref{eq:subgx}, we easily obtain 
	\begin{align*}
	\wxx^k-\beta(\xii^{k+1}-\xii^k)\in\partial g^*(\xii^{k+1})	.
	\end{align*}
	Combining the convexity of $g^*$, we get 
	\begin{align*}
	g^*(\xi^{k+1})\leq g^*(\xii^k)+\langle \xii^{k+1}-\xii^k,\wxx^k-\beta(\xii^{k+1}-\xii^k) \rangle.
	\end{align*}
	According to the update scheme \eqref{algeq:update_x}, it yields
	\begin{align*}
	&f(x^{k+1})+h^+(x^{k+1})-\langle x^{k+1}-\wxx^k,u^k \rangle+\tau\B_{\psi}(x^{k+1},\wxx^k) \nn\\
	&\leq 
	f(x^{k})+h^+(x^{k})-\langle x^{k}-\wxx^k,u^k \rangle+\tau\B_{\psi}(x^{k},\wxx^k).
	\end{align*}
	Adding the above two inequalities, and then reorganizing the resulting inequality, we have
	\begin{align*}
	&f(x^{k+1})+g^*(\xi^{k+1})+h^+(x^{k+1})-h^+(x^{k})+\langle \wxx^k,\xii^k-\xii^{k+1}\rangle+\langle x^k- x^{k+1},\xii^{k+1}+\nabla h^-(\wxx^{k}) \rangle\nn\\
	& \leq f(x^{k})+g^*(\xii^k)-\beta\|\xii^k-\xii^{k+1}\|^2+\tau\B_{\psi}(x^k,\wxx^k)-\tau\B_{\psi}(x^{k+1},\wxx^k).
	\end{align*}
	Subtracting $\langle x^{k},\xii^k\rangle$ both sides simultaneously and arranging each term, we arrive at
	\begin{align}
	&f(x^{k+1})+g^*(\xi^{k+1})-\langle x^{k+1},\xii^{k+1}\rangle+h^+(x^{k+1})-h^+(x^{k})\nn\\
	&\leq f(x^{k})+g^*(\xii^k)-\langle x^{k},\xii^{k}\rangle-\langle x^k-x^{k+1},\nabla h^-(\wxx^{k}) \rangle-\beta\|\xii^k-\xii^{k+1}\|^2 \nn\\
	&\quad+\tau\B_{\psi}(x^{k},\wxx^k)-\tau\B_{\psi}(x^{k+1},\wxx^k)+\langle x^k-\wxx^k,\xii^k-\xii^{k+1}\rangle.
	\end{align}	
	Recalling the strong convexity of $\psi$ with strongly convex modulus $\rho$, and the Lipschitz gradient continuity property with constant $L_{\psi}$, we have
	\begin{align*}
	&\tau\B_{\psi}(x^{k},\wxx^k)-\tau\B_{\psi}(x^{k+1},\wxx^k)\le \frac{\tau L_{\psi}}{2}\|x^k-\wxx^k\|^2-\frac{\tau \rho}{2}\|x^{k+1}-\wxx^k\|^2.
	\end{align*}
	From Cauchy-Schwarz inequality, there holds
	\begin{align*}
	\langle x^k-\wxx^k,\xii^k-\xii^{k+1}\rangle\leq \frac{1}{2}\|\xii^{k+1}-\xii^k\|^2+\frac{1}{2}\|\wxx^k-x^k\|^2.
	\end{align*}	
	Subsequently, we get 
	\begin{align}\label{au-decrease-2}
	&f(x^{k+1})+g^*(\xi^{k+1})-\langle x^{k+1},\xii^{k+1}\rangle+h^+(x^{k+1})-h^-(x^{k+1})\nn\\
	&-(h^+(x^{k})-h^-(x^{k}))\nn\\
	&\leq f(x^{k})+g^*(\xii^k)-\langle x^{k},\xii^{k}\rangle-\langle x^k-x^{k+1},\nabla h^-(\wxx^{k})\rangle\nn\\
	&\quad-h^-(x^{k+1})+h^-(x^{k})
	-\beta\|\xii^k-\xii^{k+1}\|^2\nn\\
	&
	\quad+\frac{\tau L_{\psi}}{2}\|x^k-\wxx^k\|^2-\frac{\tau \rho}{2}\|x^{k+1}-\wxx^k\|^2+\frac{1}{2}\|\wxx^k-x^k\|^2+\frac{1}{2}\|\xii^{k+1}-\xii^k\|^2.
	\end{align}	
	Based on the Assumption~\ref{ass:1}, it follows from Lemma \ref{lem:delem} that 
	\begin{align*}
	&-h^-(x^{k+1})+h^-(x^{k})
	\leq \langle -\nabla h^-(x^{k+1}),x^{k+1}-x^{k}\rangle + \frac{L_h^-}{2}\|x^{k+1}-x^k\|^2,
	\end{align*}
which further implies that
	\begin{align}\label{h-inequality}
	&-\langle x^k-x^{k+1},\nabla h^-(\wxx^{k})\rangle-h^-(x^{k+1})+h^-(x^{k})\nn\\
	&\leq \langle x^k-x^{k+1},\nabla h^-(x^{k+1})-\nabla h^-(\wxx^{k})\rangle + \frac{L_h^-}{2}\|x^{k+1}-x^k\|^2\nn\\
	&\leq \frac{(L_h^-)^2}{2}\|\wxx^k-x^{k+1}\|^2+\frac{L_h^-+1}{2}\|x^{k+1}-x^k\|^2,
	\end{align}
	where the last inequality holds from the Cauchy-Schwarz inequality. Based on the fact $\|x^{k+1}-x^k\|^2=\frac{1}{\alpha_{k+1}^2}\|\wxx^{k+1}-x^{k+1}\|^2$, by substituting \eqref{h-inequality} into \eqref{au-decrease-2}, we have
	\begin{align}
	&f(x^{k+1})+g^*(\xi^{k+1})-\langle x^{k+1},\xii^{k+1}\rangle+h^+(x^{k+1})-h^-(x^{k+1})\nn\\
	&\leq  f(x^{k})+g^*(\xii^k)-\langle x^{k},\xii^{k}\rangle+h^+(x^{k})-h^-(x^{k})+\left(\frac{1}{2}-\beta\right)\|\xii^k-\xii^{k+1}\|^2\nn\\
	&\quad+\frac{
		\tau    	L_{\psi}+1}{2}\|x^k-\wxx^k\|^2+\frac{(L_h^-)^2-\tau \rho}{2}\|\wxx^k-x^{k+1}\|^2+\frac{L_h^-+1}{2\alpha_{k+1}^2}\|\wxx^{k+1}-x^{k+1}\|^2.\nn
	\end{align}
	Since 
	$$\|x^{k+1}-\wxx^k\|^2= \left\|\frac{1}{\alpha_k}(\hat{x}^{k+1}-x^{k+1})+(x^k-\hat{x}^k)\right\|^2\leq \frac{2}{\alpha_{k+1}^2}\|\wxx^{k+1}-x^{k+1}\|^2+2\|x^k-\wxx^k\|^2,$$
	and for any $k\in \mathbb{N}$, $0<\alpha_{\min}\leq\alpha_k$,
	it yields
	\begin{align*}
	&f(x^{k+1})+g^*(\xi^{k+1})-\langle x^{k+1},\xii^{k+1}\rangle+h^+(x^{k+1})-h^-(x^{k+1})\nn\\
	&\leq  f(x^{k})+g^*(\xii^k)-\langle x^{k},\xii^{k}\rangle+h^+(x^{k})-h^-(x^{k})+(\frac{1}{2}-\beta)\|\xii^k-\xii^{k+1}\|^2\nn\\
	&\quad+\frac{\tau     L_{\psi}+1+2(L_h^-)^2-2\tau\rho}{2}\|x^k-\wxx^k\|^2+\frac{2(L_h^-)^2+L_h^--2\tau  \rho+1}{{2\alpha_{\min}^2}}\|\wxx^{k+1}-x^{k+1}\|^2, \nn
	\end{align*}
{which, together with the notations $H$ and $P$ in \eqref{HPeq} implies the assertion of this lemma.}
\end{proof}

 Now, we construct a merit function $\widehat{\Theta}: \mathbb{R}^n\times\mathbb{R}^n\times\mathbb{R}^n\times\mathbb{R}^n\rightarrow\bar{\mathbb{R}}$:
	\begin{align}\label{hatTheta-def}
			\widehat{\Theta}(x,\xi,\wxx,\zeta)=&~{\Theta}(x,\xi)+\delta\|x-\wxx\|^2+ \eta\|\xi-\zeta\|^2,
	\end{align}	
	For all $k\geq 1$, letting  $\left\{(x^k,\xi^k,\wxx^k,\xi^{k-1})\right\}_{k\in \mathbb{N}}$ be a sequence generated by Algorithm \ref{alg:proposed}, and setting $x=x^k, \xi=\xi^k, \wxx=\wxx^k$ and $\zeta=\xi^{k-1}$ in function (\ref{hatTheta-def}), we immediately define a new sequence as follows:
	\begin{align}\label{hatTheta}
	\left\{\widehat{\Theta}(x^k,\xi^k,\widehat{x}^k,\xi^{k-1})\right\}_{k\in \mathbb{N}}=&~\left\{{\Theta}(x^k,\xi^k)+\delta\|x^k-\wxx^k\|^2+ \eta\|\xii^{k}-\xii^{k-1}\|^2\right\}_{k\in \mathbb{N}},
	\end{align}
	where 	\begin{align*}
	\delta=\frac{\tau  L_{\psi}-L_h^-}{2[(1-\epsilon)+(1+\epsilon){\alpha_{\min}^2}]},\;\;\tau  =\frac{2(L_h^-)^2+{L_h^-}+1+{2(1+\epsilon)\delta}}{2\rho},\;\;
	\eta=\frac{2\beta-1}{2(1+\epsilon)}.
	\end{align*}
Then, let $\beta>\frac{1}{2},\;L_h^-<\tau  L_{\psi}$, and $\epsilon>0$ be an arbitrary real number satisfying $1-\epsilon>0$. We below show that  $\left\{\widehat{\Theta}(x^k,\xi^k,\widehat{x}^k,\xi^{k-1})\right\}_{k\in \mathbb{N}}$ defined by \eqref{hatTheta} satisfies following decreasing property. 
In what follows, we denote $\omega:=(x,\xi)$ and $\wo=:(x,\xi,\widehat{x},\zeta)$ for simplicity.

\begin{lemma}\label{lemma-convergence-z}
	Suppose Assumptions ~\ref{ass:1},~\ref{ass:2} and \ref{ass:3} hold, and the parameters are set as in \eqref{hatTheta}. Let $\{\wo^k\}_{k\in \mathbb{N}}:=\left\{(x^k,\xi^k,\widehat{x}^k,\xi^{k-1})\right\}_{k\in \mathbb{N}}$ be a sequence generated by Algorithm \ref{alg:proposed}. Assume that $\{\wo^k\}_{k\in \mathbb{N}}$ is bounded. Then, the following statements hold.
	\begin{itemize}
		\item[\rm (i)] $\{\widehat{\Theta}(\wo^k)\}_{k\in \mathbb{N}}$ is monotonically nonincreasing and there exists a $\sigma=\epsilon\min\{\delta,\eta\},\; $ such that 
		\begin{align}
		\label{bounded-below}
		\sigma\left(\|x^k-\wxx^k\|^2+\|x^{k+1}-\wxx^{k+1}\|^2+\|\xii^{k}-\xii^{k+1}\|^2\right) \leq \widehat{\Theta}(\wo^{k})-\widehat{\Theta}(\wo^{k+1}).
		\end{align}
		\item[\rm (ii)]
		\begin{align*}
		\sum_{k=0}^{\infty} \|x^{k+1}-\wxx^{k+1}\|^2<\infty \quad \text{and}\quad
		\sum_{k=0}^{\infty}\|\xii^{k}-\xii^{k+1}\|^2<\infty ,
		\end{align*}
		In particular, $\lim_{k\rightarrow\infty}\|\xii^{k+1}-\xii^k\|=0$, 
		and $\lim_{k\rightarrow\infty}\|x^{k+1}-\wxx^{k+1}\|=0$, which implies $\lim_{k\rightarrow\infty}\|x^k-x^{k+1}\|=0.$ 
	\end{itemize}
\end{lemma}

\begin{proof}
	\begin{itemize}
		\item[\rm (i)]
		The conclusion can be obtained immediately from Lemma \ref{phi-decrease-lemma-conclusion1} and the definition of \eqref{hatTheta} with the specific parameter settings.
		\item[\rm (ii)]	
		From Assumption \ref{ass:2} and the definition of {$\Theta$}, we know that the function {$\Theta$} is bounded below. Hence the function $\widehat{\Theta}$ is also bounded below. Consequently, it follows from (i) that $\{\widehat{\Theta}(\widehat{\omega}^k)\}$ is  monotonically nonincreasing. Therefore, $\{\widehat{\Theta}(\widehat{\omega}^k)\}$ is convergent.
		Let $N$ be a positive integer. Summing up \eqref{bounded-below} from $k=0$ to $N-1$ leads to
		\begin{align*}
		\sigma\sum_{k=0}^{N-1} \left(\|x^k-\wxx^k\|^2+\|x^{k+1}-\wxx^{k+1}\|^2+\|\xii^{k}-\xii^{k+1}\|^2\right)\leq
		{\widehat{\Theta}} (\wo^{0})-{\widehat{\Theta}} (\wo^{N}).
		\end{align*}
		From the definition of $\sigma$ and taking the limit as $N\rightarrow\infty$, immediately yields
		\begin{align*}
		\sum_{k=0}^{\infty} \left(\|x^k-\wxx^k\|^2+\|x^{k+1}-\wxx^{k+1}\|^2+\|\xii^{k}-\xii^{k+1}\|^2\right)< \infty.
		\end{align*}
		Hence $\sum_{k=0}^{\infty} \|x^{k+1}-\wxx^{k+1}\|^2<\infty.$ We can further obtain  $\lim_{k\rightarrow\infty}\|x^{k+1}-\wxx^{k+1}\|^2=0$. Based on the fact $\|x^{k+1}-x^k\|^2=\frac{1}{\alpha_{k+1}^2}\|\wxx^{k+1}-x^{k+1}\|^2$, it immediately conclude that $\lim_{k\rightarrow\infty}\|x^{k+1}-x^{k}\|^2=0$. Likewise, we have $\sum_{k=0}^{\infty}\|\xii^{k}-\xii^{k+1}\|^2<\infty $, which means $\lim_{k\rightarrow\infty}\|\xii^{k}-\xii^{k+1}\|^2=0$.
	\end{itemize}
\end{proof}

Next, we will show that the subgradient of auxiliary function $\widehat{\Theta}$ is controlled by iteration gap.
\begin{lemma}\label{lemma-subgradient}
	Suppose Assumptions~\ref{ass:1},~\ref{ass:2} and \ref{ass:3} hold, and the parameters are set as in \eqref{hatTheta}. Let $\{\wo^k\}_{k\in \mathbb{N}}$ be a sequence generated by Algorithm \ref{alg:proposed}. Assume that $\{\wo^k\}_{k\in \mathbb{N}}$ is bounded. We define $\upsilon^{k+1}\in \partial {\widehat{\Theta}}(\wo^{k+1})$, where
	\begin{align*}
	\upsilon^{k+1}=(\upsilon^{k+1}_{x},\upsilon^{k+1}_{\xii},\upsilon^{k+1}_{\wxx},\upsilon^{k+1}_{\zeta}).
	\end{align*}
	It holds
	\begin{equation}\label{subgradinent-S}
	\left\{\begin{aligned}
	\upsilon^{k+1}_{x}=&~\nabla h^-(\wxx^k)-\nabla h^-(x^{k+1})+2\delta{(x^{k+1}-\wxx^{k+1})}-\tau  (\nabla\psi(x^{k+1})-\nabla\psi(\wxx^k)),\\
	\upsilon^{k+1}_{\xi}=&~\wxx^k-x^k+\frac{1}{\alpha_{k+1}}(x^{k+1}-\wxx^{k+1})+(2\eta-\beta)(\xi^{k+1}-\xi^k),\\
	\upsilon^{k+1}_{\wxx}=&-2\delta(x^{k+1}-\wxx^{k+1}),\\
	\upsilon^{k+1}_{\zeta}=&~2\eta(\xi^{k}-\xi^{k+1}).
	\end{aligned}\right.
	\end{equation}
Then, there exists $\theta >0$ such that
	\begin{equation}\label{dist-gap-inequality}
	{\dist(0,\partial \widehat{\Theta}(\wo^{k+1}))}\leq\ \theta  \big(\|x^{k}-\wxx^{k}\|+\|x^{k+1}-\wxx^{k+1}\|+\|\xii^{k+1}-\xii^k\| \big).
	\end{equation}
\end{lemma}

\begin{proof}
	It is clear that
	\begin{equation}
	\left\{\begin{aligned}
	\partial_x {\widehat{\Theta}(\widehat{\omega}^{k+1})}=&~\partial f(x^{k+1})-\xi^{k+1}+\nabla h^+(x^{k+1})-\nabla h^-(x^{k+1})+2\delta{(x^{k+1}-\wxx^{k+1})},\nn\\
	\partial_{\xii}{\widehat{\Theta}(\widehat{\omega}^{k+1})}=&~\partial g^*(\xii^{k+1})-x^{k+1}+2\eta(\xii^{k+1}-\xii^{k}),\\
	\partial_{\wxx} {\widehat{\Theta}(\widehat{\omega}^{k+1})}=&-2\delta(x^{k+1}-\wxx^{k+1}),\\
	\partial_{\zeta} {\widehat{\Theta}(\widehat{\omega}^{k+1})}=&~2\eta(\xi^{k}-\xi^{k+1}).
	\end{aligned}\right.\nonumber
	\end{equation}
	From the first order optimality conditions of \eqref{eq:subgx}-\eqref{algeq:update_x}, we have
	\begin{align*}
	\wxx^k-\beta(\xi^{k+1}-\xi^k)\in\partial g^*(\xii^{k+1}),
	\end{align*}
	and
	\begin{align*}
	&\nabla h^-(\wxx^{k})-\nabla h^+(x^{k+1})-\tau  (\nabla\psi(x^{k+1})-\nabla\psi(\wxx^{k}))+\xii^{k+1}\in\partial f(x^{k+1}).
	\end{align*}
	Then, we can immediately get the formula \eqref{subgradinent-S}. On the other hand, from the Lipschitz gradient continuity property of {$h^{-}$} and $\psi$, we further obtain
	\begin{align*}
	\|\upsilon^{k+1}_{x}\|\leq \left({L_h^-}+\tau  L_{\psi}\right)\|x^k-{\wxx^k}\|+\left(\frac{L_h^-}{\alpha_{k+1}}+2\delta+\frac{\tau  L_{\psi}}{\alpha_{k+1}}\right)\|x^{k+1}-\wxx^{k+1}\|
	\end{align*}
	and it is evident that 
	\begin{align*}
	&\|\upsilon^{k+1}_{\xi}\|\leq \|\wxx^k-x^k\|+\frac{1}{\alpha_k}\|x^{k+1}-\wxx^{k+1}\|+(2\eta-\beta)\|\xi^{k+1}-\xi^k\|,\\
	&\|\upsilon^{k+1}_{\wxx}\|\leq 2\delta\|x^{k+1}-\wxx^{k+1}\|\quad{\rm and}\quad\|\upsilon_{\zeta}^{k+1}\|\leq 2\eta\|\xi^{k+1}-\xi^k\|.	
	\end{align*}
	Thus, there exists $\theta>0$ such that
	\begin{align*}
	\dist(0,\partial \widehat{\Theta}(\wo^{k+1}))\leq \|\upsilon^{k+1}\|=&\sqrt{\|\upsilon^{k+1}_{x}\|^2+\|\upsilon^{k+1}_{\xi}\|^2+\|\upsilon^{k+1}_{\wxx}\|^2+\|\upsilon^{k+1}_{\zeta}\|^2}\nn\\
	\leq & \|\upsilon^{k+1}_{x}\|+\|\upsilon^{k+1}_{\xi}\|+\|\upsilon^{k+1}_{\wxx}\|+\|\upsilon^{k+1}_{\zeta}\|\\
	\leq & \theta  \left(\|x^{k}-\wxx^{k}\|+\|x^{k+1}-\wxx^{k+1}\|+\|\xii^{k+1}-\xii^k\| \right),
	\end{align*}
	which completes the proof.
\end{proof}

For notational simplicity, we introduce some new symbols. For sequence $\{{\wo^k}\}_{k\in \mathbb{N}}$, we denote $\wo^{\star}:=(x^{\star},\xi^{\star},x^{\star},\xi^{\star})$ as its cluster point, and the set composed of these cluster points is defined as $\wO^{\star}$. Similarly, for $\{{\omega}^k\}_{k\in \mathbb{N}}$, let $\omega^{\star}=:(x^{\star},\xi^{\star})$ be its cluster point, and we use $\Omega^{\star}$ to represent the cluster point set of $\{{\omega}^k\}_{k\in \mathbb{N}}$.

\begin{lemma}\label{lemma-dist}
	Suppose Assumptions~\ref{ass:1},~\ref{ass:2} and \ref{ass:3} hold, and the parameters are set as in \eqref{hatTheta}. Let $\{\wo^k\}_{k\in \mathbb{N}}$ be the sequence generated by Algorithm \ref{alg:proposed}. Assume that $\{\wo^k\}_{k\in \mathbb{N}}$ is bounded. 
	 Then the following statements hold.
	\begin{itemize}
		\item[\rm (i)]
		$\Omega^{\star}$ is a nonempty compact set, and $\lim_{k\rightarrow\infty}\dist(\wo^k,\wO^{\star})=\lim_{k\rightarrow\infty}\dist({\omega}^k,\Omega^{\star})=0$.
		\item[\rm (ii)]	${\widehat{\Theta}}(\cdot)$ is a constant on $\widehat{\Omega}^{\star}$.
		\item[\rm (iii)] $\widehat{\Omega}^{\star}\subseteq {\Xi({\widehat{\Theta}})}.$
	\end{itemize}
\end{lemma}

\begin{proof} We prove the assertions of this lemma one by one.
	\begin{itemize}
		\item[\rm (i)] From the definition of $\wO^{\star}$ and $\Omega^{\star}$, it is trivial to get Item (i).
		\item[\rm (ii)] Let $\wo^{\star}\in \wO^{\star}$ then there exists a subsequence $\{\wo^{k_q}\}_{q\in \mathbb{N}}$ of $\{\wo^k\}_{k\in \mathbb{N}}$ converging to $\wo^{\star}$.
		 From the second conclusion of Lemma \ref{lemma-convergence-z}, we have already known
		\begin{align*}
		\lim_{q\rightarrow\infty}x^{k_{q}+1}=\lim_{q\rightarrow\infty}\wxx^{k_{q}+1}=x^{\star}\;\;,\;\lim_{q\rightarrow\infty}\xii^{k_q+1}=\xii^{\star}.
		\end{align*}
		Since $f$ is lower semicontinuous, $\widehat{\Theta}$ is also lower semicontinuous, which implies 
		\begin{align}\label{lsc}
		\widehat{\Theta}(\widehat{\omega}^{\star})\leq\lim_{q\rightarrow\infty}\inf \widehat{\Theta}(\widehat{\omega}^{k_q+1}).
		\end{align}
		From \eqref{eq:subgx} and \eqref{algeq:update_x}, letting $k=k_q$, it yields
		\begin{align*}
		g^*(\xi^{k_q+1}) - \langle \wxx^{k_q}, \xi^{k_q+1}\rangle +\frac{\beta}{2}\|\xi^{k_q+1}-\xi^{k_q}\|^2 \leq  g^*(\xi^{\star}) - \langle \wxx^{k_q}, \xi^{\star} \rangle +\frac{\beta}{2}\|\xi^{\star}-\xi^{k_q}\|^2 
		\end{align*}
	 and
		\begin{align*}
		&f(x^{k_q+1})+h^+(x^{k_q+1})-\langle x^{k_q+1}-\wxx^{k_q},\xii^{k_q+1}+\nabla h^-(\wxx^{k_q})\rangle+\tau\B_{\psi}(x^{k_q+1},\wxx^{k_q})\\
		 &\leq  f(x^{\star})+h^+(x^{\star})-\langle x^{\star}-\wxx^{k_q},\xii^{k_q+1}+\nabla h^-(x^{k_q})\rangle+\tau\B_{\psi}(x^{\star},\wxx^{k_q}).
		\end{align*}
	Combining the above two inequalities, adding $-h^-(x^{k_q+1})$ to the resulting inequality and reorganize it, we have
		\begin{align*}
		&f(x^{k_q+1})+g^*(x^{k_q+1})-\langle x^{k_q+1},\xii^{k_q+1}\rangle+h^{+}(x^{k_q+1})-h^-(x^{k_q+1})\\
		&	\leq   f(x^{\star})+g^*(\xi^{\star})-\langle \wxx^{\star},\xi^{k_q+1}\rangle+h^{+}(x^{\star})-h^-(x^{k_q+1})-\langle \wxx^{k_q},\xi^{\star}-\xi^{k_q+1}\rangle\\
		&\quad  -\langle x^{\star}-\wxx^{k_q+1}, \nabla h^-(x^{\star})\rangle+\frac{\beta}{2}\|\xi^{\star}-\xi^{k_q}\|^2-\frac{\beta}{2}\|\xi^{k_q+1}-\xi^{k_q}\|^2\\
		&\quad 	+\tau\B_{\psi}(x^{\star},\wxx^{k_q})-\tau\B_{\psi}(x^{k_q+1},\wxx^{k_q}).
		\end{align*}
	Taking limit in the above inequality, we arrive at
	\begin{align*}
		&\lim_{q\rightarrow\infty}\sup \widehat{\Theta}(\wo^{k_q+1})\\
		&= \lim_{q\rightarrow\infty}\sup \biggl\{ {\Theta}(\omega^{k_q+1})+\delta\|x^{k_q+1}-\wxx^{k_q+1}\|^2+ \eta\|\xii^{k_q+1}-\xii^{k_q}\|^2\biggr\}\\
		&\leq \lim_{q\rightarrow\infty}\sup \biggl\{f(x^{\star})+g^*(\xi^{\star})-\langle \wxx^{\star},\xi^{k_q+1}\rangle+h^{+}(x^{\star})-h^-(x^{k_q+1})-\langle \wxx^{k_q},\xi^{\star}-\xi^{k_q+1}\rangle\\
		&\quad  -\langle x^{\star}-\wxx^{k_q+1}, \nabla h^-(x^{\star})\rangle+\frac{\beta}{2}\|\xi^{\star}-\xi^{k_q}\|^2-\frac{\beta}{2}\|\xi^{k_q+1}-\xi^{k_q}\|^2\\
		&\quad 	+\tau\B_{\psi}(x^{\star},\wxx^{k_q})-\tau\B_{\psi}(x^{k_q+1},\wxx^{k_q})
		+\delta\|x^{k_q+1}-\wxx^{k_q+1}\|^2+ \eta\|\xii^{k_q+1}-\xii^{k_q}\|^2
		\biggr\}\\
		&= \widehat{\Theta}(\widehat{\omega}^{\star}) = {\Theta}(\omega^{\star}),
	\end{align*}
	where the second equality holds from the continuity of $h$ and Item (ii) of Lemma \ref {lemma-convergence-z}. Obviously, this inequality means
	\begin{align*}
	\lim_{q\rightarrow\infty}\sup \widehat{\Theta}(\wo^{k_q+1})\leq \widehat{\Theta}(\widehat{\omega}^{\star}),
	\end{align*}
	which, together with the fact \eqref{lsc}, yields
		\begin{align}\label{limit exist}
		\lim_{q\rightarrow\infty} {\widehat{\Theta}}(\wo^{k_q+1})={\widehat{\Theta}}(\wo^{\star})={\Theta}(\omega^{\star}).
		\end{align}
		Hence $\widehat{\Theta}(\widehat{\omega})$ is constant on $\wO^{\star}$.
		\item[\rm (iii)] By invoking Lemmas \ref{lemma-convergence-z} and \ref{lemma-subgradient}, we have 
		$\upsilon^{k+1}\in \partial \widehat{\Theta}(\wo^{k+1})$. Since $\{\widehat{\Theta}(\widehat{\omega}^k)\}$ converges uniformly on $\widehat{\Omega}^{\star}$, the order of subgradient and limit can be exchanged. Letting $k\to\infty$, it yields $\upsilon^{k+1}\to 0$, which immediately implies $0\in\partial \widehat{\Theta}(\wo^{\star})$ by the closedness of $\partial \widehat{\Theta}$. Hence $\wo^{\star}$ is a stationary point of $\widehat{\Theta}$.
	\end{itemize}
The proof is complete.
\end{proof}

	To end this subsection, we now show that $\{x^k\}_{k\in \mathbb{N}}$ converges to a stationary point of \eqref{priminal problem}.  

\begin{theorem}\label{convergence-theorem}
	Suppose $\widehat{\Theta}(\widehat{\omega})$ is a K{\L} function and Assumptions \ref{ass:1}-\ref{ass:3} hold. The parameters are set in \eqref{hatTheta}. Let $\{\omega^k\}_{k\in \mathbb{N}}$ be the sequence generated by Algorithm \ref{alg:proposed}. Assume $\{\omega^k\}_{k\in \mathbb{N}}$ is bounded. Then, the following statements hold.
	\begin{itemize}
		\item [\rm(i)] The sequence $\{\omega^k\}_{k\in \mathbb{N}}$ has finite length, i.e.,
		\begin{align*}
			\sum_{k=0}^{\infty}\|\omega^{k+1}-\omega^k\|<\infty.
		\end{align*}
	\item [\rm (ii)] The sequence $\{x^k\}_{k\in \mathbb{N}}$ converges to a critical point $x^{\star}$ of $\Phi(x)$.
	\end{itemize}
\end{theorem}

\begin{proof} The proof is divided into two parts.
\begin{itemize}
	\item [\rm (i)]
To prove Item (i), we further consider two cases based on \eqref{limit exist}. 
	\begin{itemize}
		\item[\rm (a)]If there exists a positive integer $k_1$ such that, for any $k>k_1$, ${\widehat{\Theta}}(\wo^{k})={\widehat{\Theta}}(\wo^{\star})$. Then from Item (i) of Lemma \ref{lemma-convergence-z}, it is clear that
		\begin{align*}
		&\sigma\left(\|x^k-\wxx^k\|^2+\|x^{k+1}-\wxx^{k+1}\|^2+\|\xii^{k}-\xii^{k+1}\|^2\right)
		\nonumber\\
		&\qquad \leq \widehat{\Theta}(\wo^{k})-\widehat{\Theta}(\wo^{k+1})\leq \widehat{\Theta}(\wo^{\star})-\widehat{\Theta}(\wo^{\star})=0.
		\end{align*}
		We can immediately get $x^{k+1}=\wxx^{k+1}$ for all $k>k_1$, it follows $x^{k+1}=x^k$. Then, 
		\begin{align}\label{sequence-convergence}
			\sum_{k=0}^{\infty}\|\omega^{k+1}-\omega^k\|\leq \sum_{k=0}^{\infty} \left\{\|x^{k}-x^{k+1}\|+\|\xii^k-\xii^{k+1}\|\right\}<\infty
		\end{align}
		holds for all $k$.
		Hence the conclusion is proved.

		\item[\rm (b)]Now, we assume ${\widehat{\Theta}}(\wo^k)>{\widehat{\Theta}}(\wo^{\star})$
		 for all $k>0$. Since \eqref{limit exist} implies
		that for any $\hat{\zeta}>0$, there exists a positive integer $k_2$ such that $\widehat{\Theta}(\wo^{k+1})<\widehat{\Theta}(\wo^{\star})+\hat{\zeta}$ for all $k>k_2$. It follows from Item (i) of Lemma \ref{lemma-dist} that, for any $\varsigma  >0$, there exists a positive integer $k_3$ such that $\dist(\wo^{k+1},\wO^{\star})<\varsigma$ for all $k>k_3$. Now we choose $d=\max\{k_2,k_3\}$, then for any $k>d$, we have 
		\begin{align*}
		\dist(\wo^{k+1},\wO^{\star})<\varsigma\quad \text{and}\quad \widehat{\Theta}(\wo^{\star})<\widehat{\Theta}(\wo^{k+1})<\widehat{\Theta}(\wo^{\star})+\hat{\zeta}.
		\end{align*}
		In addition, $\wO^{\star}$ is a nonempty compact set and $\widehat{\Theta}$ is constant over such a set. Then, the condition of uniformized K{\L} property (Lemma \ref{lem:UKL}) is satisfied. Consequently, we deduce that for any $k>d$,
		\begin{align}\label{KL-inequality}
		\phi^{\prime}(\widehat{\Theta}(\wo^{k+1})-\widehat{\Theta}(\wo^{\star}))\dist(0,\partial \widehat{\Theta}(\wo^{k+1}))\geq 1.
		\end{align}
		According to Lemma \ref{lemma-subgradient}, we have
		\begin{align}\label{dist-inequality-gap}
		\dist(0,\partial \widehat{\Theta}(\wo^{k+1}))\leq\ \theta  \big(\|x^{k}-\wxx^{k}\|+\|x^{k+1}-\wxx^{k+1}\|+\|\xii^{k+1}-\xii^k\| \big).
		\end{align}
		It follows from the concavity of $\phi$ that
		\begin{align}\label{concave inequality}
		&\phi\big(\widehat{\Theta}(\wo^k)-\widehat{\Theta}(\wo^{\star})\big)-\phi\big(\widehat{\Theta}(\wo^{k+1})-\widehat{\Theta}(\wo^{\star})\big)\nonumber \\
		&\;\;\geq \phi^{\prime}\big(\widehat{\Theta}(\wo^k)-\widehat{\Theta}(\wo^{\star})\big)\big(\widehat{\Theta}(\wo^k)-\widehat{\Theta}(\wo^{k+1})\big).
		\end{align}
		We denote $\triangle_k=	\phi\big(\widehat{\Theta}(\wo^k)-\widehat{\Theta}(\wo^{\star})\big)-\phi\big(\widehat{\Theta}(\wo^{k+1})-\widehat{\Theta}(\wo^{\star})\big)$ for simplicity. Associating \eqref{KL-inequality}, \eqref{dist-inequality-gap} with \eqref{concave inequality}, it follows that
		\begin{align*}
		\widehat{\Theta}(\wo^{k})-\widehat{\Theta}(\wo^{k+1})\leq\theta\triangle_k \big(\|x^{k}-\wxx^{k}\|+\|x^{k+1}-\wxx^{k+1}\|+\|\xii^{k+1}-\xii^k\| \big).
		\end{align*}
		From \eqref{bounded-below}, we can further obtain
		\begin{align}\label{sqrt-0}
		&\left(\|x^k-\wxx^k\|^2+\|x^{k+1}-\wxx^{k+1}\|^2+\|\xii^{k}-\xii^{k+1}\|^2\right)\nn\\
		&\leq \frac{\theta}{\sigma}\triangle_k \big(\|x^{k}-\wxx^{k}\|+\|x^{k+1}-\wxx^{k+1}\|+\|\xii^{k+1}-\xii^k\| \big).
		\end{align}
		Using the fact $\frac{1}{\sqrt{n}}(a_1+\dots+a_n)\leq (a_1^2+\dots+a_n^2)^{\frac{1}{2}}$ and $\sqrt{ab}\leq\frac{1}{2}(\gamma a+\frac{b}{\gamma})$ for $a,b,\gamma>0$,  we then have
		\begin{align}\label{sqrt-1}
		&\left(\frac{\theta}{\sigma}\triangle_k \big(\|x^{k}-\wxx^{k}\|+\|x^{k+1}-\wxx^{k+1}\|+\|\xii^{k+1}-\xii^k\| \big)\right)^\frac{1}{2}\nonumber\\
		&\leq\frac{1}{2}\left( \frac{\theta\gamma}{\sigma}\triangle_k+ \frac{1}{\gamma}\big(\|x^{k}-\wxx^{k}\|+\|x^{k+1}-\wxx^{k+1}\|+\|\xii^{k+1}-\xii^k\| \big)\right), 
		\end{align}
		and
		\begin{align}\label{sqrt-2}
		&{\frac{1}{\sqrt{3}}}\left(\|x^k-\wxx^k\|+\|x^{k+1}-\wxx^{k+1}\|+\|\xii^{k}-\xii^{k+1}\|\right)\nonumber\\
		&\leq \left(\|x^k-\wxx^k\|^2+\|x^{k+1}-\wxx^{k+1}\|^2+\|\xii^{k}-\xii^{k+1}\|^2\right)^{\frac{1}{2}}.
		\end{align}
	 Now, by taking the square root of both sides of inequality \eqref{sqrt-0}, and substituting \eqref{sqrt-1} and \eqref{sqrt-2} into it, we obtain
		\begin{align}\label{convergence-gap}
		\left(1-\frac{\sqrt 3}{2\gamma}\right)\left(\|x^k-\wxx^{k}\|+\|\xii^k-\xii^{k+1}\|+\|x^{k+1}-\wxx^{k+1}\|\right)\leq\frac{\sqrt 3}{2}\frac{\theta\gamma}{\sigma}\triangle_k.
		\end{align}
		Summing \eqref{convergence-gap} from $k=N$ to $\infty$ immediately yields
	\begin{align}\label{thineq}
		&\left(1-\frac{\sqrt 3}{2\gamma}\right)\sum_{k=N}^{\infty}\left(\|x^k-\wxx^{k}\|+\|\xii^k-\xii^{k+1}\|+\|x^{k+1}-\wxx^{k+1}\|\right) \nn\\
		&\leq\frac{\sqrt 3}{2}\frac{\theta\gamma}{\sigma}\phi\big(\widehat{\Theta}(\wo^N)-\widehat{\Theta}(\wo^{\star})\big).
		\end{align}
		Due to the arbitrariness of $\gamma$, we just require $\gamma>\frac{\sqrt{3}}{2}$, which immediately ensures $\left(1-\frac{\sqrt 3}{2\gamma}\right)>0$. As a consequence, the boundedness of the right hand side of \eqref{thineq} leads to
		\begin{align*}
		\left(1-{\frac{\sqrt 3}{2\gamma}}\right)\sum_{k=N}^{\infty}\big(\|x^k-\wxx^{k}\|+\|\xii^k-\xii^{k+1}\|+\|x^{k+1}-\wxx^{k+1}\|\big)<\infty,
		\end{align*}
		which means $\displaystyle\sum_{k=N}^{\infty}\|\xii^k-\xii^{k+1}\|<\infty$ and $\displaystyle\sum_{k=N}^{\infty}\|x^{k+1}-\wxx^{k+1}\|<\infty$.
		Invoking the updating scheme of $\wxx^{k}$ in \eqref{extrastep}, we have 
		 $$\wxx^{k+1}-x^{k+1}=\alpha_{k+1}(x^{k+1}-x^k),\quad \alpha_{k+1}\in (0,1].$$
		Therefore, we can easily obtain $\displaystyle\sum_{k=N}^{\infty}\|x^{k+1}-x^{k}\|<\infty$. From \eqref{sequence-convergence}, we finally draw the conclusion.
	\end{itemize}
\item [\rm (ii)] Now, we aim to show $x^\star\in\crit{\Phi}$. Since sequence $\{\omega^k\}_{k\in \mathbb{N}}$ is convergent, which implies $\{\widehat{\omega}^k\}_{k\in \mathbb{N}}$ and $\{x^k\}_{k\in \mathbb{N}}$ are also convergent. In view of Item (iii) of Lemma \ref{lemma-dist}, from $0\in\partial \widehat{\Theta}(\wo^{\star})$, we have
\begin{equation}\label{optcon}
	\left\{\begin{aligned}
		&x^{\star}\in\partial g^*(\xi^{{\star}}),\\
		&\xii^{{\star}}\in \partial f(x^{{\star}})+\nabla h^+(x^{{\star}})-\nabla h^-(x^{{\star}}).
	\end{aligned}\right.
\end{equation}
Since $x^{\star}\in\partial g^*(\xii^{{\star}})$, by the properties of conjugate functions and the continuity of $g$, we have $\xii^{\star}\in\partial g(x^{\star})$, which together with \eqref{optcon} implies that
$$\left\{\partial f(x^\star)+\nabla h(x^\star)\right\}\cap \partial g(x^\star)=\xi^\star\neq \emptyset,$$
that is,
\begin{align*}
	0\in \partial f(x^{{\star}})-\partial g(x^{\star})+\nabla h^+(x^{{\star}})-\nabla h^-(x^{{\star}}),
\end{align*}
which means that $x^{\star}$ is a critical point of \eqref{priminal problem}.
\end{itemize}
The proof is complete.
\end{proof}

\section{Numerical experiments}\label{Sec:Num}

In this section, we apply Algorithm \ref{alg:proposed} (iBPDCA for short) to low-rank matrix and tensor completion problems, and compare it with the classical DCA (i.e., iterative scheme \eqref{DCA-one} \cite{PDS86}) and the BPDCA (i.e., iBPDCA without the  extrapolation \eqref{extrastep}). All codes were written by Matlab 2021b and all experiments were conducted on a laptop computer with Intel (R) core (TM) i7-7500 CPU @ 2.70GHz and 8GB memory.

\subsection{Low-rank matrix completion}
In this subsection, we consider the following low-rank matrix completion problem:
\begin{align}\label{matrix-comp-model}
\min_{X \in \RR^{m\times n}}~\lambda(\|X\|_*-\|X\|_F)+\frac{1}{2}\|\mathscr{P}_{\Omega}(X-M)\|_F^2,
\end{align}
where $M$ is an incompleted matrix, $\Omega$ is a subset of the index set of entries $\{1,\dots,m\}\times\{1,\dots,n\}$, $\mathscr{P}_{\Omega}(\cdot):\RR^{m\times n}\to\RR^{m\times n}$ is the orthogonal projection so that the $ij$-th entry of $\mathscr{P}_{\Omega}(M)$ is $M_{ij}$ if $(i,j)\in\Omega$ and zero otherwise.

We can  easily reformulate model (\ref{matrix-comp-model}) as a special case of ($\ref{priminal problem}$) by setting $x=X$, $$f(X)=\lambda\|X\|_*,~g({X}) = \lambda\|X\|_F,~h^+({X})=\frac{1}{2}\|\mathscr{P}_{\Omega}(X-M)\|_F^2, \;\; \text{and}\;\; h^-({X})=0.$$
Due to the appearance of $\mathscr{P}_\Omega$, we take the Bregman kernel function as $\psi_k(\cdot) = \frac{1}{2}\|\cdot\|_{M_k}^2$, where $M_k = \mu_k I-\mathscr{P}_\Omega^\top\mathscr{P}_\Omega$ with some appropriate $\mu_k >0$ such that $M_k$ is positive definite.
Throughout this subsection, we set $\mu_k=1.1,\lambda = 0.5, \beta=1$ and $\tau=1$. At each iteration, we update $\alpha_k=\frac{t^{k-1}-1}{t^k}$, where $t^k=\frac{1}{2}(1+\sqrt{1+4(t^{k-1})^2})$ and $ t^0=1$. When implementing the classical DCA and BPDCA, the underlying $h(X)$ is directly set as  $h(X)=\frac{1}{2}\|\mathscr{P}_{\Omega}(X-M)\|_F^2$. Then, applying the DCA \eqref{DCA-one} to \eqref{matrix-comp-model} yields
\begin{equation}\label{DCA-MC}
X^{k+1}=\arg\min_{X}\left\{ \lambda \|X\|_* -\langle \xi^{k+1},X-X^k\rangle+\frac{1}{2}\|\mathscr{P}_{\Omega}(X-M)\|_F^2\right\}
\end{equation}	
where $\xi^{k+1}\in\partial(\lambda\|X^k\|_F)$. However, we notice that the subproblem \eqref{DCA-MC} has no
closed-form solution, which is not practically feasible. Therefore, we approximately obtain $X^{k+1}$ via the powerful alternating direction method of multipliers (ADMM \cite{GM76,GM75}), where $h(X)$ is further transformed into $h(Y)$ via an auxiliary variable $Y$ so that $\|X\|_*$ and $h(Y)$ are separable. In our experiments, we set the penalty parameter as $1.1^j$ in the $\{k,j\}$-th outer-inner iteration, and take $\|X^{k,{j}}-Y^j\|_F\leq 10^{-3}$ as the stopping criterion for ADMM, where $k$ and $j$ count the outer and inner iterations, respectively. For simplicity, we denote this algorithm by ADMM-DCA for short. The iterative scheme of BPDCA for \eqref{matrix-comp-model} reads as 
\begin{align*}
X^{k+1}&=\arg\min_{X}\left\{ \lambda \|X\|_* +\langle \nabla h(X^k)-\xi^{k+1},X-X^k\rangle+\frac{\mu_k}{2}\|X-X^k\|_F^2\right\} \\
&=\mathscr{D}_{\frac{\lambda}{\mu_k}}\left(X^k-\frac{1}{\mu_k}\left(\nabla h(X^k)-\xi^{k+1}\right)\right),
\end{align*}	
where $\xi^{k+1}$ is updated by scheme (\ref{eq:subgx}) with $\widehat{X}^k = X^k$ and $\mathscr{D}_\kappa(\cdot)$ is the well-known singular value shrinkage operator, which is defined by $\mathscr{D}_\kappa(A)=U\mathscr{D}_\kappa(\Sigma)V^\top$ with $A=U\Sigma V^\top$ being the SVD of a matrix $A\in\mathbb{R}^{m\times n}$ of rank $r$ in the reduced form, i.e., $\Sigma={\rm diag}(\{\sigma_i\}_{1\leq i\leq r})$, $U$ and $V$ are orthogonal matrices, and $\mathscr{D}_\kappa(\Sigma)=\text{diag}(\max\{\sigma_i-\kappa,0\})$ for $\kappa\geq 0$.

We first consider some synthetic datasets to evaluate the feasibility and reliability of iBPDCA. More concretely,  we generate a low-rank matrix $X_{\rm true}$ as a ground truth matrix, which follows $X_{\rm true} = UV+0.01\cdot{\rm randn}(m,n)$, where $U\in\RR^{m\times r}$, $V\in\RR^{r\times n}$ with the entries being random samples drawn from a uniform distribution, and $r$ is much less than $\min\{m,n\}$. In this way, we construct a low-rank matrix whose rank is less than $r$. In our experiments, we take $r=10$, and set $m=n=100,500$ and $1000$, respectively. Moreover, we apply Matlab script $\Omega = {\rm rand}(m,n)<{\rm sample~ratio}$ (sr), and then set the observed matrix $M=\mathscr{P}_{\Omega}(X_{\rm true})$. Throughout this subsection, we set ${\rm sr} = 0.2$ and $0.5$. We define
$$
{\rm RSE} = \frac{\|X^{*}-X_{\rm true}\|}{\|X_{\rm true}\|},
$$ to measure the quality of the recovered matrix. Moreover, we report the computing time in seconds (Time(s)), number of iteration (Iter) and the rank of the recovered matrix (rank) in Table \ref{synthetic-result}. We see that three algorithms can achieve the almost same RSE and rank, but our iBPDCA takes the least time in all cases. 
\begin{table}
	\centering
	\caption{Computational results of solving problem (\ref{matrix-comp-model}) with synthetic data}\label{synthetic-result}
	\resizebox{1\columnwidth}{!}{
		\small\begin{tabular}{cc c c c c c c c c c c c c c c c}\Xhline{1.2pt}
			\multicolumn{1}{c}{\multirow{2}{*}{sr}}&\multicolumn{2}{c}{\multirow{2}{*}{size}}&\multicolumn{4}{c}{BPDCA}&&\multicolumn{4}{c}{iBPDCA} &&\multicolumn{4}{c}{ADMM-DCA}\\
			\cmidrule{4-7}\cmidrule{9-12}\cmidrule{14-17}
			&\multicolumn{2}{c}{}&RSE&Time(s)&Iter&rank&&RSE&Time(s)&Iter&rank &&RSE&Time(s)&Iter&rank  \\
			\Xhline{1pt}
			\multicolumn{1}{c}{\multirow{2}{*}{0.5}}&	\multicolumn{2}{c}{(100,100)}&1.63e-02&0.26&121&10&&1.62e-02&0.21&87&10 &&1.73e-02&0.60&15&10\\
			&\multicolumn{2}{c}{(500,500)}&2.40e-03&13.49&192&10&&2.41e-03&5.44&76&10 &&2.45e-03&39.71&20&10\\
			&\multicolumn{2}{c}{(1000,1000)}&1.19e-03&185.00&262&10 &&1.19e-03&59.49&84&10  &&1.21e-03&540.16&27&10\\
			\Xhline{1pt}
			\multicolumn{1}{c}{\multirow{2}{*}{0.2}}&\multicolumn{2}{c}{(100,100)}&6.29e-02&0.93&387&10&&6.28e-02&0.56&207&10 &&6.48e-02&1.87&50&10\\
			&\multicolumn{2}{c}{(500,500)}&7.39e-03&33.38&432&10&&7.37e-03&10.48&124&10 &&8.47e-03&100.53&48&10\\
			&\multicolumn{2}{c}{(1000,1000)}&2.35e-02&316.92&500&10 &&3.21e-03&99.21&147&10  &&3.52e-03&1023.96&62&10\\
			\Xhline{1.2pt}
	\end{tabular}}
\end{table}

	\begin{figure}
	\includegraphics[width=1\textwidth]{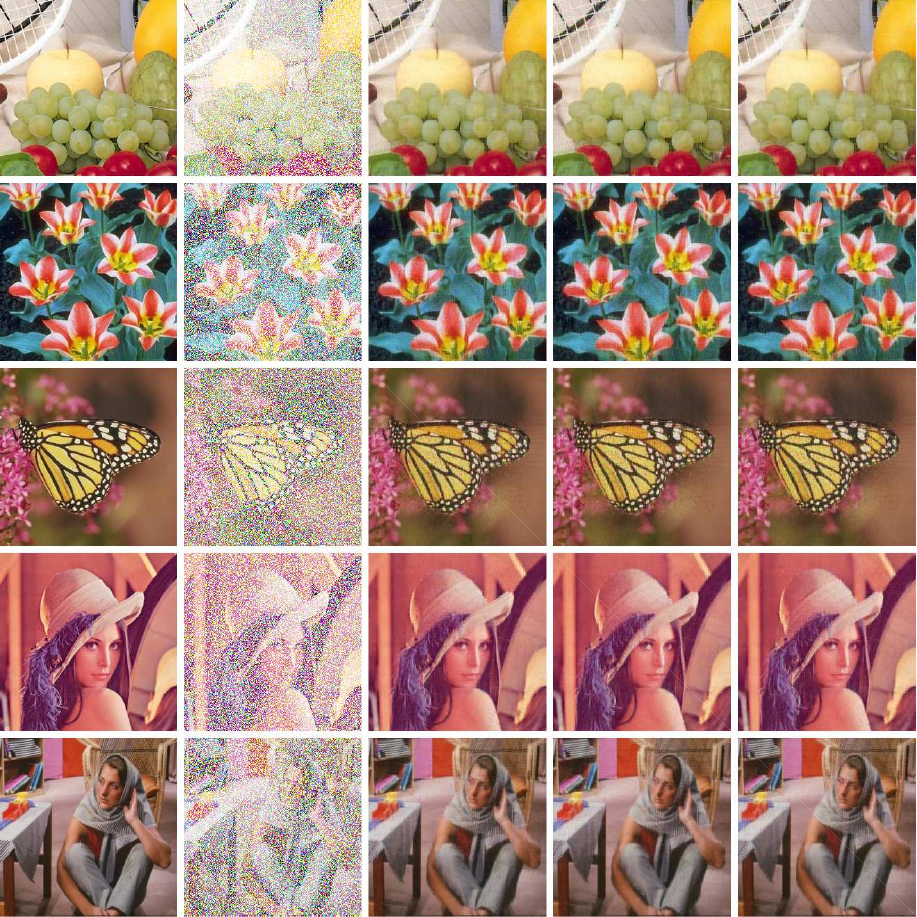}
	\caption{Images recovered by BPDCA, iBPDCA and {ADMM-DCA}. From top to bottom: fruits, tulips, butterfly, lena, and barbara. From left to right: original images, downsampled images with sr$=0.5$, image recoverd by BPDCA, iBPDCA and {ADMM-DCA}, respectively.}
	\label{image-show}
\end{figure}

\begin{table}
	\centering
	\caption{Computational results of solving problem (\ref{matrix-comp-model}) with color images}\label{image-result}
	
	\resizebox{1\columnwidth}{!}{
		\small\begin{tabular}{cc c c c c c c c c c c c c c c c}\Xhline{1.2pt}
			\multicolumn{1}{c}{\multirow{2}{*}{sr}}&\multicolumn{2}{c}{\multirow{2}{*}{image}}&\multicolumn{4}{c}{BPDCA}&&\multicolumn{4}{c}{iBPDCA} &&\multicolumn{4}{c}{ADMM-DCA}\\
			\cmidrule{4-7}\cmidrule{9-12}\cmidrule{14-17}
			&\multicolumn{2}{c}{}&PSNR&Time(s)&Iter&rank&&PSNR&Time(s)&Iter&rank &&PSNR&Time(s)&Iter&rank  \\
			\Xhline{1pt}
			&\multicolumn{2}{c}{fruits}&27.78&2.94&103&120 &&27.79&1.81&63&120  &&27.67&8.53&15&126\\
			\multicolumn{1}{c}{\multirow{2}{*}{0.5}}&	\multicolumn{2}{c}{tulips}&25.93&3.63&127&143 &&25.93&2.11&74&143  &&25.81&10.80&19&152\\
			&\multicolumn{2}{c}{buttergly}&24.90&2.63&137&119 &&24.95&1.58&80&119  &&24.87&8.99&20&126\\
			&\multicolumn{2}{c}{lena}&28.72&2.41&85&113 &&28.73&1.85&65&113 &&28.60&7.26&13&120\\
			&\multicolumn{2}{c}{barbara}&28.18&2.76&91&120 &&28.19&2.02&70&120 &&28.07&7.90&14&128\\
			\Xhline{1pt}
			&\multicolumn{2}{c}{fruits}&21.66&6.85&242&79 &&21.67&2.91&103& 79 &&21.60&19.90&34&100\\
			\multicolumn{1}{c}{\multirow{2}{*}{0.2}}&	\multicolumn{2}{c}{tulips}&21.75&8.37&238&94 &&21.75&2.50&86& 94 &&21.69&20.64&33&118\\
			&\multicolumn{2}{c}{butterfly}&19.12&8.39&295&76 &&19.12&4.20&115&76 &&19.09&28.85&43&96\\
			&\multicolumn{2}{c}{lena}&22.70&6.14&197&73 &&22.70&3.01&102&73 &&22.62&17.82&30&92\\
			&\multicolumn{2}{c}{barbara}&22.13&8.05&233&78 &&22.14&3.65&125&78 &&22.05&19.11&32&99\\
			\Xhline{1.2pt}
	\end{tabular}}
\end{table}



Hereafter, we are concerned with the numerical performance of our algorithm on real-world datasets. In our experiments, we choose five widely used color images (`fruits', `tulips', `butterfly', `lena' and `barbara') with size of  $256\times 256\times 3$. Here, we convert the original image $\X_{\rm true}\in\RR^{256\times 256\times 3}$ into a matrix $X_{\rm true}\in\RR^{256\times 768}$. To report the results, we further employ the Peak Signal-to-Noise Ratio (PSNR) defined by 
\begin{align*}
\text{PSNR} = 10\log_{10}\frac{(X_{\rm true}^{\max})^2(\#\Omega^C)}{\|X^*-X_{\text{true}}\|_F^2},
\end{align*}
to measure the quality of the recovered image, where $X_{\rm true}^{\max}$ is the maximum element of $X_{\rm true}$, where $\#\Omega^C$ denotes the number of elements
in the complementary set of $\Omega$. All results are summarized in Table \ref{image-result} and Figure \ref{image-show}.  It can be seen from Figure \ref{image-show} that the recovered images by three algorithms have almost the same quality. However, we can see from Table \ref{image-result} that BPDCA and iBPDCA perform better than {ADMM-DCA} in terms of computing time and PSNR values.

\subsection{Low-tubal-rank tensor completion}
The objective variable in model (\ref{matrix-comp-model}) is a matrix. In this subsection, we consider a more general case where the variable is a tensor in the objective function. Before our experiments, we briefly recall some basic concepts of the tensor nuclear norm (see \cite{HLX22,KM11} for more details).  We represent the tensor $\bar{\X}$ in terms of $\X$ after performing the Fast Fourier Transform (FFT) on each tube, i.e., $\bar{\X} = {\rm fft}(\X,[],3)$ and  $\X = {\rm ifft}(\bar{\X},[],3)$ in {\sc {\sc Matlab}}. We use $X_k$ (resp. $\bar{X_k}$) to denote the $k$-th frontal slice of a third-order tensor $\X\in\mathbb{R}^{n_1\times n_2\times n_3}$ (resp. $\bar{\X}\in\mathbb{R}^{n_1\times n_2\times n_3})$\footnote{ $\X(:,:,k)$ (resp. $\bar{\X}(:,:,k)$) in Matlab}.
With the above preparations, we consider the following low-tubal-rank tensor completion problem:
\begin{align}\label{tensor-comp-model}
\min_{\X\in\mathbb{R}^{n_1\times n_2\times n_3}}~\lambda(\|\X\|_{\rm TNN}-\|\X\|_F)+\frac{1}{2}\|\mathscr{P}_{\Omega}(\X-\mathcal{M})\|_F^2,
\end{align}
where $\|\X\|_{\rm TNN}: = \frac{1}{n_3}\sum_{k=1}^{n_3}\sum_{i=1}^{\min\{n_1,n_2\}}\sigma_i(\bar{X_k})$, and $  \sigma_i(\bar{X_k})$ is the $i$-th singular value of $\bar{X_k}$.  Clearly, model (\ref{tensor-comp-model}) can also be regarded as a special case of ($\ref{priminal problem}$) by setting $x=\mathscr{X}$, 
$$f({\X})=\lambda\|\X\|_*,~g({\X}) = \lambda\|\X\|_F,~h^+({\X})=0,\;\; \text{and}\;\; h^-({\X})=-\frac{1}{2}\|\mathscr{P}_{\Omega}(\X-\mathcal{M})\|_F^2.$$
The main difference between problem (\ref{tensor-comp-model}) and (\ref{matrix-comp-model}) is that the objective variable is changed from a matrix $X$ to a tensor $\X$, so the corresponding tensor nuclear norm $\|\cdot\|_{\rm TNN}$ is considered. All parameters for iBPDCA for solving (\ref{tensor-comp-model}) are same to the matrix case. Similar to model (\ref{matrix-comp-model}), we also employ the BPDCA and {ADMM-DCA} to solve problem (\ref{tensor-comp-model}). The iterative scheme of BPDCA for solving (\ref{tensor-comp-model}) can be written as follow:
\begin{align*}
\X^{k+1}&=\arg\min_{\X}\left\{ \lambda \|\X\|_{\rm TNN} +\langle \nabla h(\X^k)-\xi^{k+1},\X-\X^k\rangle+\frac{\mu_k}{2}\|\X-\X^k\|_F^2\right\} \\
&=\mathcal{U}*{\rm ifft}\left(\widehat{\mathscr{D}}_{\frac{\lambda}{\mu_k}}(\Gamma),[],3\right)*\mathcal{V}^{\hat{\top}},
\end{align*}	
where $\xi^{k+1}$ is updated by scheme (\ref{eq:subgx}) with $\widehat{\X}^k = \X^k, \X^k-\frac{1}{\mu_k}\left(\nabla h(\X^k)-\xi^{k+1}\right)=\mathcal{U}*\Gamma*\mathcal{V}^{\hat{\top}}$\footnote{$\X^{\hat{\top}}\in\RR^{n_1\times n_2\times n_3}$ is the conjugate transpose for a tensor $\X\in\RR^{n_1\times n_2\times n_3}$, obtained by conjugate transposing each of the frontal slices of $\X$ and then reversing the order of transposed frontal slices 2 through $n_3$.}, which is the tensor singular value decomposition (see \cite{KM11} for more details) with tubal rank $r$, where $\mathcal{U}\in\RR^{n_1\times r\times n_3}$ and $\mathcal{V}\in\RR^{r\times n_2\times n_3}$ are othogonal tensors and ${\Gamma}\in\RR^{r\times r\times n_3}$ is an f-diagonal tensor, and $\widehat{\mathscr{D}}_{\kappa}(\Gamma)\in\RR^{r\times r\times n_3}$ is defined by
\begin{align*}
 	\widehat{\mathscr{D}}_{\kappa}(\Gamma)_{iik}=\max\{\bar{\Gamma}_{iik}-\kappa,0\},~i = 1,\dots,r;~k=1,\dots,n_3,
 \end{align*} 
 where $\kappa\geq 0.$
When applying the DCA \eqref{DCA-one} to solve (\ref{tensor-comp-model}), we have 
\begin{equation*}
\X^{k+1}=\arg\min_{X}\left\{ \lambda \|\X\|_{\rm TNN} -\langle \xi^{k+1},\X-\X^k\rangle+\frac{1}{2}\|\mathscr{P}_{\Omega}(\X-\mathscr{M})\|_F^2\right\}
\end{equation*}	
where $\xi^{k+1}\in\partial(\lambda\|\X^k\|_F)$.  Like the matrix case, we also introduce an auxiliary variable $\mathcal{Y}=\X$ to reformulate $h(\X)$ as $h(\mathcal{Y})$. Then, we obtain $\X^{k+1}$ by the ADMM with same settings used in \eqref{DCA-MC}.

We start to apply our iBPDCA to solving problem (\ref{tensor-comp-model}) with synthetic datasets. Here, we randomly generate low-tubal-rank $\X\in\RR^{n_1\times n_2\times n_3}$ and $\Omega$ by the following way. First, we obtain $\mathscr{U}={\rm randn}(n_1,r,n_3)$ and $\mathscr{V}={\rm randn}(r,n_2,n_3)$ by Matlab. Then, we let $\X=U*V+0.01\cdot{\rm randn}(n_1,n_2,n_3)$. For $\Omega$, we apply Matlab script ${\rm rand}(n_1,n_2,n_3)<{\rm sr}$. In this part, we set $r = 5$, $n_1=n_2=\{20,50,100\}$ and $n_3=10$. In Table \ref{synthetic-result-T}, we report RSE, Time(s), Iter and the Tubal rank (Trank) of using three algorithms to solve model (\ref{tensor-comp-model}). It is easy to see that iBPDCA is faster than the other two algorithms in terms of computing time. In addition, iBPDCA can achieve lower tubal rank than {ADMM-DCA}, which, to some extent, means that iBPDCA can recover better solutions. 

\begin{table}
	\centering
	\caption{Computational results of solving problem (\ref{tensor-comp-model}) with synthetic data}\label{synthetic-result-T}
	\resizebox{1\columnwidth}{!}{
		\small\begin{tabular}{cc c c c c c c c c c c c c c c c}\Xhline{1.2pt}
			\multicolumn{1}{c}{\multirow{2}{*}{sr}}&\multicolumn{2}{c}{\multirow{2}{*}{size}}&\multicolumn{4}{c}{BPDCA}&&\multicolumn{4}{c}{iBPDCA} &&\multicolumn{4}{c}{ADMM-DCA}\\
			\cmidrule{4-7}\cmidrule{9-12}\cmidrule{14-17}
			&\multicolumn{2}{c}{}&RSE&Time(s)&Iter&Trank&&RSE&Time(s)&Iter&Trank &&RSE&Time(s)&Iter&Trank  \\
			\Xhline{1pt}
			\multicolumn{1}{c}{\multirow{2}{*}{0.5}}&	\multicolumn{2}{c}{(20,20,10)}&1.66-02&0.89&654&10&&1.64e-02&0.40&226&10 &&1.66e-02&1.51&82&11\\
			&\multicolumn{2}{c}{(50,50,10)}&3.35-03&2.53&443&6&&3.33e-03&0.76&146&5 &&3.58e-03&5.03&48&7\\
			&\multicolumn{2}{c}{(100,100,10)}&1.56e-03&11.77&600&5&&1.55e-03&3.07&161&5&&1.61e-03&27.65&62&5\\
			\Xhline{1pt}
			\multicolumn{1}{c}{\multirow{2}{*}{0.2}}&\multicolumn{2}{c}{(20,20,10)}&4.75e-02&1.27&867&6&&4.79e-02&0.37&215&6 &&4.75e-02&2.32&116&8\\
			&\multicolumn{2}{c}{(50,50,10)}&3.25e-02&5.58&1090&14&&3.32e-02&1.43&281&14 &&3.31e-02&12.47&141&18\\
			&\multicolumn{2}{c}{(100,100,10)}&7.39e-03&57.49&2945&22&&5.36e-03&7.55&375&19 &&6.01e-03&134.52&354&26\\
			\Xhline{1.2pt}
	\end{tabular}}
\end{table}

Finally, we are interested in recovering grayscale videos from incomplete observations. Here, we choose four grayscale videos\footnote{http://trace.eas.asu.edu/yuv/.} including `Airport', `Lobby', `Akiyo' and `Hall'. Some preliminary numerical results are summarized in Table \ref{video-result}. We can see that iBPDCA always outperforms BPDCA and {ADMM-DCA} which further verifies the efficiency of our algorithm. In addition,  some recovered frames by three algorithms are shown in Figure \ref{video-show}, which show that all algorithms can achieve satisfied results.  

\begin{table}
	\centering
	\caption{Computational results of solving problem (\ref{tensor-comp-model}) with video datasets}\label{video-result}
	\resizebox{1\columnwidth}{!}{
		\small\begin{tabular}{cc c c c c c c c c c c c c c c c}\Xhline{1.2pt}
			\multicolumn{1}{c}{\multirow{2}{*}{sr}}&\multicolumn{2}{c}{\multirow{2}{*}{image}}&\multicolumn{4}{c}{BPDCA}&&\multicolumn{4}{c}{iBPDCA} &&\multicolumn{4}{c}{ADMM-DCA}\\
			\cmidrule{4-7}\cmidrule{9-12}\cmidrule{14-17}
			&\multicolumn{2}{c}{}&PSNR&Time(s)&Iter&Trank&&PSNR&Time(s)&Iter&Trank &&PSNR&Time(s)&Iter&Trank  \\
			\Xhline{1pt}
			&\multicolumn{2}{c}{airport}&29.54&22.42&148&116 &&29.63&11.51&71&116 &&29.63&81.68&23&117\\
			\multicolumn{1}{c}{\multirow{2}{*}{0.5}}&	\multicolumn{2}{c}{lobby}&40.06&12.18&98&107 &&40.10&7.12&59&107 &&40.04&36.19&13&109\\
			&\multicolumn{2}{c}{akiyo}&41.13&14.32&96&101 &&41.24&8.92&60&101  &&41.11&46.59&13&102\\
			&\multicolumn{2}{c}{hall}&40.31&13.42&92&122 &&40.44&8.45&57&122  &&40.36&47.26&13&122\\
			\Xhline{1pt}
			&\multicolumn{2}{c}{airport}&23.85&45.92&308&84 &&23.96&16.75&104&84 &&23.96&174.49&53&89\\
			\multicolumn{1}{c}{\multirow{2}{*}{0.2}}&	\multicolumn{2}{c}{lobby}&35.42&29.92&254&92 &&35.60&10.79&65&92 &&35.32&90.68&33&93\\
			&\multicolumn{2}{c}{akiyo}&34.61&32.87&220&76 &&34.84&20.44&122	&76  &&34.67&107.37&30&78\\
			&\multicolumn{2}{c}{hall}&34.38&30.75&202&99 &&34.66&18.13&110&99  &&34.48&102.40&29&100\\
			\Xhline{1.2pt}
	\end{tabular}}
\end{table}

\begin{figure}
	\includegraphics[width=1\textwidth]{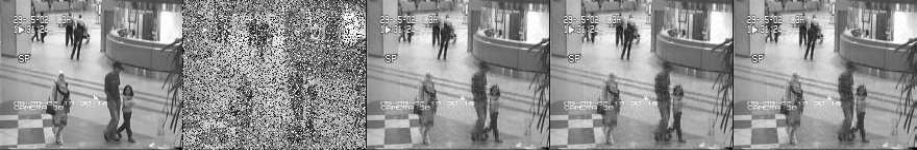}
	\includegraphics[width=1\textwidth]{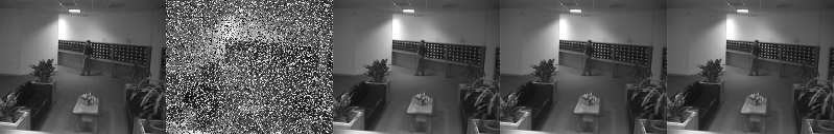}
	\includegraphics[width=1\textwidth]{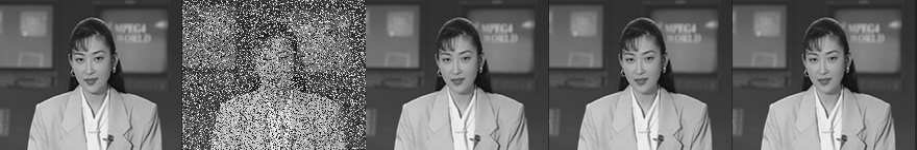}
	\includegraphics[width=1\textwidth]{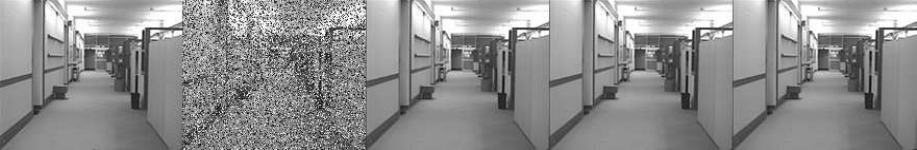}
	\caption{Some frames recovered by BPDCA, iBPDCA and {ADMM-DCA}. From top to bottom: airport, lobby, akiyo, and hall. From left to right: original frames, downsampled frames with sr$=0.5$, frames recovered by BPDCA, iBPDCA, and {ADMM-DCA}, respectively.}
	\label{video-show}
\end{figure}

\section{Conclusion}\label{Sec:Con} 
In this paper, we introduce a new efficient DC algorithm for solving a class of generalized DC programming. With the help of the Fenchel-Young inequality, we constructed a well-defined subproblem in the sense a unique solution can be determined iteratively, thereby circumventing the difficulty of selecting an ideal subgradient for algorithmic implementation. Moreover, we can gainfully exploit the possibly explicit form of the proximal operator of $g(x)$ by the extended Moreau decomposition theorem. The imposed Bregman proximal term in the $x$-subproblem also makes our algorithm versatile for designing customized algorithms. Some computational results on matrix/tensor completion problems demonstrate that our algorithm performs well in practice.

\vskip 6mm
\noindent{\bf Acknowledgements}

\noindent
The authors would like to thank the referees for their valuable comments, which helped us improve the presentation of this paper.
H.J. He was supported in part by Zhejiang Provincial Natural Science Foundation of China (No. LZ24A010001),  Ningbo Natural Science Foundation (Project ID: 2023J014) and National Natural Science Foundation of China  (No. 12371303) and. C. Ling was supported in part by National Natural Science Foundation of China (No. 11971138).


\begin{thebibliography}{10}
 	\providecommand{\url}[1]{{#1}}
 	\providecommand{\urlprefix}{URL }
 	\expandafter\ifx\csname urlstyle\endcsname\relax
 	\providecommand{\doi}[1]{DOI~\discretionary{}{}{}#1}\else
 	\providecommand{\doi}{DOI~\discretionary{}{}{}\begingroup
 		\urlstyle{rm}\Url}\fi
 	
 	\bibitem{AEB06}
 	Aharon, M., Elad, M., Bruckstein, A.: K-svd: An algorithm for designing
 	overcomplete dictionaries for sparse representation.
 	\newblock IEEE Trans. Signal Process \textbf{54}(11), 4311--4322 (2006)
 	
 	\bibitem{ABRS10}
 	Attouch, H., Bolte, J., Redont, P., Soubeyran, A.: Proximal alternating
 	minimization and projection methods for nonconvex problems: an approach based
 	on the {K}urdyka-{L}ojasiewicz inequality.
 	\newblock Math. Oper. Res. \textbf{35}, 438--457 (2010)
 	
 	\bibitem{Beck17}
 	Beck, A.: First-order methods in optimization.
 	\newblock SIAM, Philadelphia (2017)
 	
 	\bibitem{BT09}
 	Beck, A., Teboulle, M.: A fast iterative shrinkage-thresholding algorithm for
 	linear inverse problems.
 	\newblock SIAM J. Imaging Sci. \textbf{2}(1), 183--202 (2009)
 	
 	\bibitem{BST14}
 	Bolte, J., Sabach, S., Teboulle, M.: Proximal alternating linearized
 	minimization for nonconvex and nonsmooth problems.
 	\newblock Math. Program. Ser. A \textbf{146}, 459--494 (2014)
 	
 	\bibitem{Bregman-1967}
 	Bregman, L.: The relaxation method of finding the common point of convex sets
 	and its application to the solution of problems in convex programming.
 	\newblock U.S.S.R. Computational Math. Math. Phys. \textbf{7}, 200--217 (1967)
 	
 	\bibitem{CHZ22}
 	Chuang, C.S., He, H., Zhang, Z.: A unified {D}ouglas--{R}achford algorithm for
 	generalized {DC} programming.
 	\newblock J. Global Optim. \textbf{82}, 331--349 (2022)
 	
 	\bibitem{GM76}
 	Gabay, D., Mercier, B.: A dual algorithm for the solution of nonlinear
 	variational problems via finite element approximations.
 	\newblock Comput. Math. Appl. \textbf{2}, 16--40 (1976)
 	
 	\bibitem{GM75}
 	Glowinski, R., Marrocco, A.: Approximation par \'{e}l\'{e}ments finis d'ordre
 	un et r\'{e}solution par p\'{e}nalisation-dualit\'{e} d'une classe de
 	probl\`{e}mes non lin\'{e}aires.
 	\newblock R.A.I.R.O. \textbf{R2}, 41--76 (1975)
 	
 	\bibitem{GTT17}
 	Gotoh, J., Takeda, A., Tono, K.: {DC} formulations and algorithms for sparse
 	optimization problems.
 	\newblock Math. Progr. Ser. B. \textbf{169}, 141--176 (2017)
 	
 	\bibitem{HLX22}
 	He, H., Ling, C., Xie, W.: Tensor completion via a generalized transformed
 	tensor t-product decomposition without t-{SVD}.
 	\newblock J. Sci. Comput. \textbf{93}, Article No. 47 (35 pages) (2022)
 	
 	\bibitem{KM11}
 	Kilmer, M., Martin, C.: Factorization strategies for third-order tensors.
 	\newblock Linear Algebra Appl. \textbf{435}(3), 641--658 (2011)
 	
 	\bibitem{LTPD15}
 	Le~Thi, H., Pham~Dinh, T.: Feature selection in machine learning: an exact
 	penalty approach using a difference of convex function algorithm.
 	\newblock Mach. Learn. \textbf{101}, 163--186 (2015)
 	
 	\bibitem{LTPD18}
 	{Le Thi}, H., {Pham Dinh}, T.: {DC} programming and {DCA}: Thirty years of
 	developments.
 	\newblock Math. Program. Ser. A \textbf{169}, 5--68 (2018)
 	
 	\bibitem{LPT19}
 	Liu, T., Pong, T., Takeda, A.: A refined convergence analysis of {pDCA}$_e$
 	with applications to simultaneous sparse recovery and outlier detection.
 	\newblock Comput. Optim. Appl. \textbf{73}, 69--100 (2019)
 	
 	\bibitem{LZ19}
 	Lu, Z., Zhou, Z.: Nonmonotone enhanced proximal {DC} algorithms for structured
 	nonsmooth {DC} programming.
 	\newblock SIAM J. Optim. \textbf{29}, 2725--2752 (2019)
 	
 	\bibitem{LZS19}
 	Lu, Z., Zhou, Z., Sun, Z.: Enhanced proximal {DC} algorithms with extrapolation
 	for a class of structured nonsmooth {DC} minimization.
 	\newblock Math. Program. Ser. B \textbf{176}, 369--401 (2019)
 	
 	\bibitem{OC15}
 	O{'}Donoghue, B., Cand{\'{e}}s, E.: Adaptive restart for accelerated gradient
 	schemes.
 	\newblock Found. Comput. Math. \textbf{15}, 715--732 (2015)
 	
 	\bibitem{deO20}
 	de~Oliveira, W.: The {ABC} of {DC} programming.
 	\newblock Set-Valued Var. Anal. \textbf{28}, 679--706 (2020)
 	
 	\bibitem{PDLT97}
 	Pham~Dinh, T., Le~Thi, H.: Convex analysis approach to {DC} programming:
 	Theory, algorithms and applications.
 	\newblock Acta Math. Vietnamica \textbf{22}, 289--355 (1997)
 	
 	\bibitem{PDS86}
 	{Pham Dinh}, T., Souad, E.B.: Algorithms for solving a class of nonconvex
 	optimization problems. {M}ethods of subgradients.
 	\newblock In: J.B. Hiriart-Urruty (ed.) Fermat Days 85: Mathematics for
 	Optimization, \emph{North-Holland Mathematics Studies}, vol. 129, pp. 249 --
 	271. North-Holland (1986)
 	
 	\bibitem{PLT18}
 	Phan, D., Le, H., Thi, H.: Accelerated difference of convex functions algorithm
 	and its application to sparse binary logistic regression.
 	\newblock IJCAI'18: Proceedings of the 27th International Joint Conference on
 	Artificial Intelligence pp. 1369--1375 (2018)
 	
 	\bibitem{PLT23}
 	Phan, D., {Le Thi}, H.: Difference-of-convex algorithm with extrapolation for
 	nonconvex, nonsmooth optimization problems.
 	\newblock Math. Oper. Res.  (2023).
 	\newblock \doi{10.1287/moor.2020.0393}
 	
 	\bibitem{RW98}
 	Rockafellar, R., Wets, R.: Variational Analysis.
 	\newblock Springer-Verlag (1998)
 	
 	\bibitem{SHWJ23}
 	Shan, Y., Hu, D., Wang, Z., Jia, T.: Multi-channel nuclear norm minus frobenius
 	norm minimization for color image denoising.
 	\newblock Signal Process \textbf{207}, 108,959 (2023)
 	
 	\bibitem{SSC03}
 	Sun, W., Sampaio, R., Candido, M.: Proximal point algorithm for minimization of
 	{DC} functions.
 	\newblock J. Comput. Math. \textbf{21}, 451--462 (2003)
 	
 	\bibitem{TFT22}
 	Takahashi, S., Fukuda, M., Tanaka, M.: New {B}regman proximal type algorithm
 	for solving {DC} optimization problems.
 	\newblock Comput. Optim. Appl. \textbf{83}, 893--931 (2022)
 	
 	\bibitem{WCP18}
 	Wen, B., Chen, X., Pong, T.K.: A proximal difference-of-convex algorithm with
 	extrapolation.
 	\newblock Comput. Optim. Appl. \textbf{69}(2), 297--324 (2018)
 	
 	\bibitem{YLHX15}
 	Yin, P., Lou, Y., He, Q., Xin, J.: Minimization of $\ell_{1-2}$ for compressed
 	sensing.
 	\newblock SIAM J. Sci. Comput. \textbf{37}, A536--A563 (2015)
 	
 \end{thebibliography}

\end{document}